\documentclass[oneside,english]{amsart}
\usepackage[T1]{fontenc}
\usepackage[latin9]{inputenc}
\usepackage{amsthm}
\usepackage{amstext}
\usepackage{amssymb}
\usepackage{cite}
\usepackage{esint}
\usepackage{float}
\usepackage{graphicx}
\usepackage{multirow}
\usepackage{subfig}

\usepackage{url}
\usepackage{tikz}
\usepackage{pgfplots}
\usetikzlibrary{decorations.pathmorphing,patterns}

\floatstyle{plaintop}
\restylefloat{table}

 \makeatletter

\makeatletter

\numberwithin{equation}{section}
\numberwithin{figure}{section}
\theoremstyle{plain}
\newtheorem{thm}{\protect\theoremname}[section]
  \theoremstyle{definition}
  \newtheorem{defn}[thm]{\protect\definitionname}
  \theoremstyle{plain}
  \newtheorem{prop}[thm]{\protect\propositionname}
  \theoremstyle{remark}
  
  \theoremstyle{plain}
  
  \theoremstyle{plain}
  
  \theoremstyle{plain}
  \newtheorem{lem}[thm]{\protect\lemmaname}

\makeatletter
\@addtoreset{thm}{section}
\makeatother

\makeatother

\usepackage{babel}
  \providecommand{\corollaryname}{Corollary}
  \providecommand{\definitionname}{Definition}
  \providecommand{\factname}{Fact}
  \providecommand{\lemmaname}{Lemma}
  \providecommand{\propositionname}{Proposition}
  \providecommand{\remarkname}{Remark}
  
\providecommand{\theoremname}{Theorem}

\begin{document}

\title[Torsion Discriminance for Stability of Linear Time-Invariant Systems]
{Torsion Discriminance for Stability of Linear Time-Invariant Systems}

\author[Y. Wang]{Yuxin Wang $^1$}
\author[H. Sun]{Huafei Sun $^{1,*}$}
\author[Y. Cao]{Yueqi Cao $^1$}
\author[S. Zhang]{Shiqiang Zhang $^1$}

\thanks{This subject is supported by the National Natural Science Foundations of China (No. 61179031.)}
\thanks{$^1$ School of Mathematics and Statistics, Beijing Institute of Technology, Beijing 100081, P.~R.~China}
\thanks{$^*$ Huafei Sun is the corresponding author}
\thanks{E-mail: wangyuxin@bit.edu.cn,~ huafeisun@bit.edu.cn,~ bityueqi@gmail.com,~ zsqbit@126.com}

\begin{abstract}
This paper proposes a new approach to describe the stability of linear time-invariant systems via the torsion $\tau(t)$ of the state trajectory. For a system $\dot{r}(t)=Ar(t)$ where $A$ is invertible, we show that
(1) if there exists a measurable set $E_1$ with positive Lebesgue measure,
such that $r(0)\in E_1$ implies that
$\lim\limits_{t\to+\infty}\tau(t)\neq0$ or $\lim\limits_{t\to+\infty}\tau(t)$ does not exist,
then the zero solution of the system is stable;
(2) if there exists a measurable set $E_2$ with positive Lebesgue measure,
such that $r(0)\in E_2$ implies that
$\lim\limits_{t\to+\infty}\tau(t)=+\infty$,
then the zero solution of the system is asymptotically stable.
Furthermore, we establish a relationship between the $i$th curvature $(i=1,2,\cdots)$ of the trajectory and the stability of the zero solution when $A$ is similar to a real diagonal matrix.
\end{abstract}

\keywords{linear systems, stability, asymptotic stability, torsion, curvature}

\subjclass[2010]{53A04 93C05 93D05 93D20}

\maketitle

\section{Introduction}
\par
It is well known that Lyapunov \cite{Lyapunov} laid the foundation of stability theory. Linear systems are the most basic and widely used research objects, which have been developed for a long period. However, the traditional methods rely heavily on linear algebra. There are few results obtained from geometric aspects. 

\par
Curvature and torsion are important concepts in differential geometry. In \cite{Wang}, Wang et al. gave a description of the stability for two- and three-dimensional linear time-invariant systems $\dot{r}(t)=Ar(t)$ by calculating the curvature and torsion of the state trajectory $r(t)$ in each case. Furthermore, in \cite{Wang2} the authors use the definition of higher curvatures of curves in $\mathbb{R}^n$ given in \cite{Gluck} to obtain the relationship between the first curvature of the state trajectory and the stability of the $n$-dimensional linear system.

\par
In this paper, we will describe the stability of the zero solution of $n$-dimensional linear time-invariant system by using the torsion, namely, the second curvature.

\par
Our main results are as follows.

\begin{thm} \label{thm main 1}
Suppose that $\dot{r}(t)=Ar(t)$ is a linear time-invariant system, where $A$ is similar to an $n\times n$ real diagonal matrix, $r(t)\in\mathbb{R}^n$, and $\dot{r}(t)$ is the derivative of $r(t)$.
Denote by $\kappa_i(t)\,(i=1,2,\cdots)$ the $i$th curvature of trajectory of a solution $r(t)$. We have
\par
$(1)$ if there exists a measurable set $E\subseteq \mathbb{R}^n$ whose Lebesgue measure is greater than $0$, such that $r(0)\in E$ implies that $\lim\limits_{t\to+\infty}\kappa_i(t)\neq0$ or $\lim\limits_{t\to+\infty}\kappa_i(t)$ does not exist, then the zero solution of the system is stable;
\par
$(2)$ if $A$ is invertible, then under the assumptions of $(1)$, the zero solution of the system is asymptotically stable.
\end{thm}

\begin{thm} \label{thm main 2}
Suppose that $\dot{r}(t)=Ar(t)$ is a linear time-invariant system, where $A$ is an $n\times n$ invertible real matrix, and $r(t)\in\mathbb{R}^n$.
Denote by $\tau(t)$ the torsion of trajectory of a solution $r(t)$. We have
\par
$(1)$ if there exists a measurable set $E_1\subseteq \mathbb{R}^n$ whose Lebesgue measure is greater than $0$, such that $r(0)\in E_1$ implies that $\lim\limits_{t\to+\infty}\tau(t)\neq0$ or $\lim\limits_{t\to+\infty}\tau(t)$ does not exist, then the zero solution of the system is stable;
\par
$(2)$ if there exists a measurable set $E_2\subseteq \mathbb{R}^n$ whose Lebesgue measure is greater than $0$, such that $r(0)\in E_2$ implies that $\lim\limits_{t\to+\infty}\tau(t)=+\infty$, then the zero solution of the system is asymptotically stable.
\end{thm}

\par
The paper is organized as follows.
In Section \ref{Section Preliminaries}, we review some basic concepts and propositions.
In Section \ref{Section Diagonal}, we study the relationship between the $i$th curvature $(i=1,2,\cdots)$ of the trajectory and the stability of the zero solution of the system when the system matrix is similar to a real diagonal matrix, and we prove Theorem \ref{thm main 1}.
In Section \ref{Section Torsion}, we establish a relationship between the torsion of the trajectory and the stability of the zero solution of the system, and complete the proof of Theorem \ref{thm main 2}.
Two examples are given in Section \ref{Section Examples}.
Finally, Section \ref{Section Conclusion} concludes the paper.

\section{Preliminaries}\label{Section Preliminaries}
Throughout this paper, all vectors will be written as column vectors,
and $\|x\|$ will denote the Euclidean norm of $x=(x_1,x_2,\cdots,x_n)^\mathrm{T}\in\mathbb{R}^n$, namely, $\|x\|=\sqrt{\sum_{i=1}^n x_i^2}$.
The vector $r^{(i)}(t)$ denotes the $i$th derivative of vector $r(t)$.
We denote by $\det A$ the determinant of matrix $A$.
The eigenvalues of matrix $A$ are denoted by $\lambda_i(A)\,(i=1,2,\cdots,n)$,
and the set of eigenvalues of matrix $A$ is denoted by $\sigma(A)$.
The degree of polynomial $f(t)$ is denoted by $\deg(f(t))$.

\subsection{Stability of Linear Time-Invariant Systems}
%
%\begin{defn}[\!\!\cite{Perko}]
%The system of equations
%\begin{align}
%\left\{
%\begin{aligned}\label{system0}
%\dot{r}(t)&=Ar(t)+Bu(t), \\
%y(t)&=Cr(t)+Du(t),
%\end{aligned}
%\right.
%\end{align}
%is called a linear time-invariant system, where $r(t), u(t)$, and $y(t)$ are vectors in $\mathbb{R}^n$, $\mathbb{R}^k$, and $\mathbb{R}^l$ respectively, vector $\dot{r}(t)$ is the derivative of $r(t)$, and matrices $A_{n\times n}, B_{n\times k}, C_{l\times n}$, and $D_{l\times k}$ are real constant matrices.
%\end{defn}
%In particular, we will focus on the system
%\begin{align}\label{system1}
%\dot{r}(t)=Ar(t)
%\end{align}
%in Section \ref{Section Diagonal} and Section \ref{Section Torsion}.

\begin{defn}[\!\!\cite{Perko}]
The system of ordinary differential equations
\begin{align}\label{system1}
\dot{r}(t)=Ar(t)
\end{align}
is called a linear time-invariant system, where $A$ is an $n\times n$ real constant matrix, $r(t)\in\mathbb{R}^n$, and $\dot{r}(t)$ is the derivative of $r(t)$.
\end{defn}

\begin{prop}[\!\!\cite{Perko}]\label{ODE}
The initial value problem
\begin{align}
\left\{
\begin{aligned}\label{system}
\dot{r}(t)&=Ar(t),\\
r(0)&=r_0,
\end{aligned}
\right.
\end{align}
has a unique solution given by
\begin{align}\label{align solution}
r(t)=\mathrm{e}^{tA}r_0,
\end{align}
where $\mathrm{e}^{tA}=\sum_{k=0}^\infty {\frac{t^k A^k}{k!}}$.
\par
The curve $r(t)$ is called the trajectory of the system (\ref{system}) with the initial value $r_0\in\mathbb{R}^n$.
\end{prop}

\begin{defn}[\!\!\cite{Chen,Marsden}]
The solution $r(t)\equiv0$ of differential equations (\ref{system1}) is called the zero solution of the linear time-invariant system.
If for every constant $\varepsilon >0$,
there exists a $\delta=\delta(\varepsilon)>0$,
such that $\|r(0)\|<\delta $ implies that $\|r(t)\|<\varepsilon$ for all $t\in [0,+\infty)$,
where $r(t)$ is a solution of (\ref{system1}),
then we say that the zero solution of system (\ref{system1}) is stable.
If the zero solution is not stable, then we say that it is unstable.
\par
Suppose that the zero solution of system (\ref{system1}) is stable,
and there exists a $\tilde \delta\,(0<\tilde \delta\leqslant\delta)$,
such that $\|r(0)\|<\tilde \delta$ implies that $\lim\limits_{t\to+\infty}r(t)=0$,
then we say that the zero solution of system (\ref{system1}) is asymptotically stable.
\end{defn}

\begin{prop}[\!\!\cite{Chen}]\label{prop asy.stable}
The zero solution of system (\ref{system1}) is stable if and only if all eigenvalues of matrix $A$ have nonpositive real parts and those eigenvalues with zero real parts are simple roots of the minimal polynomial of $A$.
\par
The zero solution of system (\ref{system1}) is asymptotically stable if and only if all eigenvalues of matrix $A$ have negative real parts, namely,
$\mathrm{Re}\{\lambda_i(A)\}<0\,(i=1,2,\cdots,n)$.
\end{prop}

\begin{prop}[\!\!\cite{Chen}]
Suppose that $A$ and $B$ are two $n\times n$ real matrices, and $A$ is similar to $B$, namely, there exists an $n\times n$ real invertible matrix $P$, such that $A=P^{-1}BP$. For system (\ref{system1}), let $v(t)=Pr(t)$. Then the system after the transformation becomes
\begin{align}\label{system2}
\dot{v}(t)=Bv(t).
\end{align}
System (\ref{system2}) is said to be equivalent to system (\ref{system1}), and $v(t)=Pr(t)$ is called an equivalence transformation.
\end{prop}

\begin{prop}[\!\!\cite{Chen}]\label{prop similar}
Let $A$ and $B$ be two $n\times n$ real matrices, and $A$ is similar to $B$. Then the zero solution of the system $\dot{r}(t)=Ar(t)$ is (asymptotically) stable if and only if the zero solution of the system $\dot{v}(t)=Bv(t)$ is (asymptotically) stable.
\end{prop}

%\begin{defn}[\!\!\cite{Chen}]
%An input $u(t)$ of system (\ref{system0}) is said to be bounded if there exists a constant $\beta_1$ such that
%\begin{align*}
%\|u(t)\| \leqslant \beta_1 <+\infty, \quad \left.\forall t\in\left[t_0,+\infty\right).\right.
%\end{align*}
%Similarly, an output $y(t)$ is said to be bounded if there exists a constant $\beta_2$ such that
%\begin{align*}
%\|y(t)\| \leqslant \beta_2 <+\infty, \quad \left.\forall t\in\left[t_0,+\infty\right).\right.
%\end{align*}
%System (\ref{system0}) is said to be BIBO stable (bounded-input bounded-output stable) if every bounded
%input excites a bounded output.
%\end{defn}
%
%\begin{prop}[\!\!\cite{Chen}]%\label{prop BIBO}
%If the zero solution of system (\ref{system1}) is asymptotically stable, then system (\ref{system0}) is BIBO stable.
%\end{prop}

\subsection{Curvatures of Curves in $\mathbb{R}^n$}

\begin{defn}[\!\!\cite{doCarmo}]
Let $r:[0, +\infty)\to\mathbb{R}^3$ be a smooth curve. The functions
\begin{align*}
\kappa(t)=\frac{\left\|\dot{r}(t)\times\ddot{r}(t)\right\|}{\left\|\dot{r}(t)\right\|^3}, \quad
\tau(t)=\frac{\left(\dot{r}(t),\ddot{r}(t),\dddot{r}(t)\right)}{\left\|\dot{r}(t)\times\ddot{r}(t)\right\|^2}
\end{align*}
are called the curvature and torsion of the curve $r(t)$, respectively.
\end{defn}

\par
Gluck \cite{Gluck} gave a definition of higher curvatures of curves in $\mathbb{R}^n$,
which is a generalization of curvature and torsion.
Here we omit the definition of higher curvatures and review their calculation formulas directly.

\par
In this paper, $V_i(t)$ denotes the $i$-dimensional volume of the $i$-dimensional parallelotope with vectors
$\dot{r}(t)$, $\ddot{r}(t)$, $\cdots$, $r^{(i)}(t)$
as edges, and we have a convention that $V_0(t)=1$.

\begin{prop}[\!\!\cite{Gluck}]\label{prop curvature}
Let $r:[0, +\infty)\to\mathbb{R}^n$ be a smooth curve, and $\dot r(t)\neq0$ for all $t\in[0, +\infty)$.
Suppose that for each $t\in[0, +\infty)$, the vectors
$\dot r(t), \ddot r(t), \cdots, r^{(m)}(t)\,(m\leqslant n)$
are linearly independent. %where $r^{(m)}(t)$ denotes the $m$th derivative of $r(t)$.
Then the $i$th curvature of a curve $r(t)$ is
\begin{align*}
\kappa_i(t)
=\frac{ V_{i-1}(t) V_{i+1}(t) }{ V_{1}(t) V_{i}^2(t) } \quad
(i=1,2,\cdots,m-1).
\end{align*}
\end{prop}

\par
In \cite{Gluck}, according to the definition of the curvatures of curves in $\mathbb{R}^n$, we have
$\kappa_i(s)\geqslant0$ for $i=1,2,\cdots,m-1$.

\par
If $r(t)$ is a smooth curve in $\mathbb{R}^3$, and
$\dot r(t), \ddot r(t), \dddot r(t)$ are linearly independent, then we have Frenet-Serret formulas (cf. \cite{doCarmo}), where $\kappa_1(s)=\kappa(s)$, and $\kappa_2(s)=|\tau(s)|$, which means the first and second curvature are the generalization of curvature and torsion of curves in $\mathbb{R}^3$, respectively.
In the remainder of this paper, we use $\kappa(t)$ instead of $\kappa_1(t)$, and $\tau(t)$ instead of $\kappa_2(t)$, for simplicity.

\par
We can give $V_i(t)$ by the derivatives of $r(t)$ with respect to $t$. In fact, we have the following result.

\par
\begin{prop}[\!\!\cite{Wang2}]\label{prop V}
Write $r^{(i)}(t)=\left( r_1^{(i)}(t), r_2^{(i)}(t), \cdots, r_n^{(i)}(t) \right)^{\mathrm{T}}$.
We have
\begin{align*}
V_k^2(t)=
\sum_{1\leqslant i_1<i_2<\cdots <i_k\leqslant n}
\begin{vmatrix}
\dot{r}_{i_1}(t) & \ddot{r}_{i_1}(t) & \cdots & r_{i_1}^{(k)}(t) \\[1em]
\dot{r}_{i_2}(t) & \ddot{r}_{i_2}(t) & \cdots & r_{i_2}^{(k)}(t) \\[1em]
\vdots & \vdots & \ddots & \vdots \\[1em]
\dot{r}_{i_k}(t) & \ddot{r}_{i_k}(t) & \cdots & r_{i_k}^{(k)}(t)
\end{vmatrix}^2.
\end{align*}
\end{prop}

\par
By Proposition \ref{prop curvature} and \ref{prop V}, we obtain the expression of each curvature of curve $r(t)$ in $\mathbb{R}^n$ by the coordinates of derivatives of $r(t)$.
In particular, if $\dot{r}(t)$ and $\ddot{r}(t)$ are linearly independent, then the torsion of $r(t)$ satisfies

\begin{align}\label{align torsion}
\tau(t)
%=\kappa_2(t)
=\frac{ V_{3}(t) }{ V_{2}^2(t) }
=\frac{
\sqrt{
 \sum_{1\leqslant i<j<k\leqslant n}
\begin{vmatrix}
\dot{r}_{i}(t) & \ddot{r}_{i}(t) & \dddot{r}_{i}(t) \\[1em]
\dot{r}_{j}(t) & \ddot{r}_{j}(t) & \dddot{r}_{j}(t) \\[1em]
\dot{r}_{k}(t) & \ddot{r}_{k}(t) & \dddot{r}_{k}(t)
\end{vmatrix}^2
}
}
{
 \sum_{1\leqslant p<q\leqslant n}
\begin{vmatrix}
\dot{r}_{p}(t) & \ddot{r}_{p}(t) \\[1em]
\dot{r}_{q}(t) & \ddot{r}_{q}(t)
\end{vmatrix}^2
}.
\end{align}
On the other hand, if $V_{2}(t)\equiv 0$, namely $\dot{r}(t)$ and $\ddot{r}(t)$ are linearly dependent for all $t$, then obviously we have the convention that $\tau(t)\equiv 0$. Further, the function $V_{2}(t)$ will be examined in detail in Subsection \ref{subsection V2}.

\subsection{Relationship Between the Curvatures of Two Equivalent Systems}\label{Section Relationship}

\

Wang et al. \cite{Wang2} establish a relationship between the curvatures of the trajectories of two equivalent systems.
In fact, let a curve $r(t)$ be the trajectory of system (\ref{system}),
and suppose that for each $t$, the vectors
$\dot{r}(t), \ddot{r}(t), \cdots, r^{(m)}(t)$
are linearly independent.
Then we can define curvatures $\kappa_{r,1}(t), \kappa_{r,2}(t), \cdots, \kappa_{r,m-1}(t)$ of the curve $r(t)$, and we have the following result.

\begin{prop}[\!\!\cite{Wang2}]\label{prop relation}
Suppose that a linear time-invariant system $\dot{r}(t)=Ar(t)$ is equivalent to a system $\dot{v}(t)=Bv(t)$, where
 $A=P^{-1}BP$, and $v(t)=Pr(t)$ is the equivalence transformation.
 Let $\kappa_{r,i}(t)$ and $\kappa_{v,i}(t)$ be the $i$th $(i=1,2,\cdots,m-1)$ curvatures of trajectories $r(t)$ and $v(t)$, respectively. Then we have
\begin{align*}
\lim\limits_{t\to+\infty}\kappa_{r,i}(t)=0 &\iff \lim\limits_{t\to+\infty}\kappa_{v,i}(t)=0, \\
\lim\limits_{t\to+\infty}\kappa_{r,i}(t)=+\infty &\iff \lim\limits_{t\to+\infty}\kappa_{v,i}(t)=+\infty, \\
\kappa_{r,i}(t) ~\text{is a bounded function} &\iff \kappa_{v,i}(t) ~\text{is a bounded function}.
\end{align*}
\end{prop}

\subsection{Real Jordan Canonical Form}

\begin{prop}[\!\!\cite{Marsden,Horn}]\label{prop Jordan1}
Let $A$ be an $n\times n$ real matrix. Then $A$ is similar to a block diagonal real matrix
\begin{align}\label{real Jordan matrix}
\begin{pmatrix}
 \begin{matrix}
  C_{n_1}(a_1,b_1)&&\\[0.7ex]
  &C_{n_2}(a_2,b_2)&\\
  &&\ddots
 \end{matrix}
 &&\text{\LARGE$0$}\\
 &C_{n_p}(a_p,b_p)&\\[0.7ex]
 \text{\LARGE$0$}&&
 \begin{matrix}
  J_{n_{p+1}}(\lambda_{p+1})&&\\
  &\ddots&\\
  &&J_{n_{r}}(\lambda_r)
 \end{matrix}
\end{pmatrix},
\end{align}
where
\par
$(1)$ for $k\in\{1,2,\cdots,p\}$, the numbers $\lambda_k=a_k+\sqrt{-1} b_k$ and $\bar\lambda_k=a_k-\sqrt{-1} b_k$
\ $(a_k,b_k\in\mathbb{R},\text{and}\ b_k>0)$ are complex eigenvalues of $A$, and
\begin{align*}
C_{n_k}(a_k,b_k)=
\begin{pmatrix}
\Lambda_k & I_2&&&\\[0.6em]
&\Lambda_k & I_2&&\\
&&\Lambda_k &\ddots&\\
&&&\ddots & I_2\\[0.5em]
&&&&\Lambda_k
\end{pmatrix}_{2n_k\times 2n_k},
\end{align*}
where
$\Lambda_k=
\begin{pmatrix}
 a_k&b_k\\
-b_k&a_k
\end{pmatrix},
I_2=
\begin{pmatrix}
1&0\\
0&1
\end{pmatrix};$
\par
$(2)$ for $j\in\{p+1,p+2,\cdots,r\}$, the number $\lambda_j$ is a real eigenvalue of $A$, and
\begin{align*}
J_{n_j}(\lambda_j)=
\begin{pmatrix}
\lambda_j &1&&&\\[0.6em]
&\lambda_j &1&&\\
&&\lambda_j &\ddots&\\
&&&\ddots &1\\[0.5em]
&&&&\lambda_j
\end{pmatrix}_{n_j\times n_j}.
\end{align*}
The matrix (\ref{real Jordan matrix}) is called the real Jordan canonical form of $A$.
\end{prop}

\section{Real Diagonal Matrix}\label{Section Diagonal}

In this section, we study the case that the system matrix is similar to a real diagonal matrix, and prove Theorem \ref{thm main 1}. From Proposition \ref{prop similar} and Proposition \ref{prop relation}, we only need to focus on the case that $A$ is a real diagonal matrix, and prove Proposition \ref{prop diagonal}.

In what follows, %write $r(0)=(r_{1}(0),r_{2}(0),\cdots,r_{n}(0))^\mathrm{T}$, and
we defind a subset of $\mathbb{R}^n$ that
\begin{align*}
S=\left\{ r(0) \Bigg| r(0)=(r_{1}(0),r_{2}(0),\cdots,r_{n}(0))^\mathrm{T}\in\mathbb{R}^n, \mathrm{s.t.}\ \prod_{i=1}^n {r_{i}(0)}\neq0 \right\}.
\end{align*}

\begin{prop} \label{prop diagonal}
Suppose that $\dot{r}(t)=Ar(t)$ is a linear time-invariant system, where $A$ is an $n\times n$ real diagonal matrix, and $r(t)\in\mathbb{R}^n$. Denote by $\kappa_i(t)\,(i=1,2,\cdots)$ the $i$th curvature of trajectory of a solution $r(t)$. Then for any given initial value $r(0)\in S$, %\mathbb{R}^n, ~\mathrm{s.t.,}~ \prod_{i=1}^n {r_{i0}}\neq0$,
we have
\par
$(1)$ if $\lim\limits_{t\to+\infty}\kappa_i(t)\neq0$ or $\lim\limits_{t\to+\infty}\kappa_i(t)$ does not exist, then the zero solution of the system is stable;
\par
$(2)$ if $A$ is invertible, and $\lim\limits_{t\to+\infty}\kappa_i(t)\neq0$ or $\lim\limits_{t\to+\infty}\kappa_i(t)$ does not exist, then the zero solution of the system is asymptotically stable.
\end{prop}

Wang et al. \cite{Wang2} has proved the case of $i=1$. Now we give a complete proof of this proposition.

\begin{proof}
$(1)$ Suppose that $A$ is an $n\times n$ real diagonal matrix, namely,
\begin{align*}
A=\mathrm{diag}\{ \lambda_1, \lambda_2, \cdots, \lambda_n \}.
\end{align*}
Then
\begin{align*}
A^k=\mathrm{diag}\{\lambda_1^k, \lambda_2^k, \cdots, \lambda_n^k \}, \quad
\mathrm{e}^{tA}=
\mathrm{diag}\{\mathrm{e}^{\lambda_1 t}, \mathrm{e}^{\lambda_2 t}, \cdots, \mathrm{e}^{\lambda_n t} \},
\end{align*}
where $p=1,2,\cdots$.
Hence we have
\begin{align*}
& r(t)=\mathrm{e}^{tA}r(0)
=\left( \mathrm{e}^{\lambda_1 t}r_{1}(0), \mathrm{e}^{\lambda_2 t}r_{2}(0), \cdots, \mathrm{e}^{\lambda_n t}r_{n}(0) \right)^\mathrm{T},
\\& \dot{r}(t)=Ar(t)
=\left( \lambda_1 \mathrm{e}^{\lambda_1 t}r_{1}(0), \lambda_2 \mathrm{e}^{\lambda_2 t}r_{2}(0), \cdots, \lambda_n \mathrm{e}^{\lambda_n t}r_{n}(0) \right)^\mathrm{T},
\\& \cdots\cdots,
\\& r^{(k)}(t)=A^k r(t)
=\left( \lambda_1^k \mathrm{e}^{\lambda_1 t}r_{1}(0), \lambda_2^k \mathrm{e}^{\lambda_2 t}r_{2}(0), \cdots, \lambda_n^k \mathrm{e}^{\lambda_n t}r_{n}(0) \right)^\mathrm{T},
\end{align*}
namely, the coordinates of derivatives of $r(t)$ are
\begin{align*}
\dot{r}_i(t)=\lambda_i \mathrm{e}^{\lambda_i t}r_{i}(0), \quad
\cdots\cdots, \quad
r^{(k)}_i(t)=\lambda_i^k \mathrm{e}^{\lambda_i t}r_{i}(0) \quad
(i=1,2,\cdots,n).
\end{align*}
Then by Proposition \ref{prop V}, we obtain
\begin{align}\label{Vk^2}
V_k^2(t)&=
\sum_{1\leqslant i_1<i_2<\cdots <i_k\leqslant n}
\begin{vmatrix}
\dot{r}_{i_1}(t) & \ddot{r}_{i_1}(t) & \cdots & r_{i_1}^{(k)}(t) \\[1em]
\dot{r}_{i_2}(t) & \ddot{r}_{i_2}(t) & \cdots & r_{i_2}^{(k)}(t) \\[1em]
\vdots & \vdots & \ddots & \vdots \\[1em]
\dot{r}_{i_k}(t) & \ddot{r}_{i_k}(t) & \cdots & r_{i_k}^{(k)}(t)
\end{vmatrix}^2
\nonumber
\displaybreak[0]
\\&=
\sum_{1\leqslant i_1<i_2<\cdots <i_k\leqslant n}
\begin{vmatrix}
\lambda_{i_1} \mathrm{e}^{\lambda_{i_1} t}r_{i_1}(0)
& \lambda_{i_1}^2 \mathrm{e}^{\lambda_{i_1} t}r_{i_1}(0)
& \cdots
& \lambda_{i_1}^k \mathrm{e}^{\lambda_{i_1} t}r_{i_1}(0)
\\[1em]
\lambda_{i_2} \mathrm{e}^{\lambda_{i_2} t}r_{i_2}(0)
& \lambda_{i_2}^2 \mathrm{e}^{\lambda_{i_2} t}r_{i_2}(0)
& \cdots
& \lambda_{i_2}^k \mathrm{e}^{\lambda_{i_2} t}r_{i_2}(0)
\\[1em]
\vdots & \vdots & \ddots & \vdots \\[1em]
\lambda_{i_k} \mathrm{e}^{\lambda_{i_k} t}r_{i_k}(0)
& \lambda_{i_k}^2 \mathrm{e}^{\lambda_{i_k} t}r_{i_k}(0)
& \cdots
& \lambda_{i_k}^k \mathrm{e}^{\lambda_{i_k} t}r_{i_k}(0)
\end{vmatrix}
^2
\nonumber
\displaybreak[0]
\\&=
\sum_{1\leqslant i_1<i_2<\cdots <i_k\leqslant n}
\left\{
\mathrm{e}^{\left(\sum_{p=1}^k \lambda_{i_p}\right) t}
\prod_{q=1}^k\left(\lambda_{i_q} r_{i_q}(0)\right)
\begin{vmatrix}
1
& \lambda_{i_1}
& \lambda_{i_1}^2
& \cdots
& \lambda_{i_1}^{k-1}
\\[0.7em]
1
& \lambda_{i_2}
& \lambda_{i_2}^2
& \cdots
& \lambda_{i_2}^{k-1}
\\[0.7em]
\vdots & \vdots & \vdots & \ddots & \vdots \\[0.7em]
1
& \lambda_{i_k}
& \lambda_{i_k}^2
& \cdots
& \lambda_{i_k}^{k-1}
\end{vmatrix}
\right\}^2
\nonumber
\displaybreak[0]
\\&=
\sum_{1\leqslant i_1<i_2<\cdots <i_k\leqslant n}
\mathrm{e}^{2\left(\sum_{p=1}^k \lambda_{i_p}\right) t}
\left\{
\prod_{q=1}^k\left(\lambda_{i_q} r_{i_q}(0)\right)
\prod_{1\leqslant\alpha<\beta\leqslant k}\left(\lambda_{i_\beta}-\lambda_{i_\alpha}\right)
\right\}^2.
\end{align}
We see that if the eigenvalues $\lambda_{i_1}, \lambda_{i_2}, \cdots, \lambda_{i_k}$ of $A$ are non-zero and distinct, then a term of the form $C\mathrm{e}^{2\left(\sum_{p=1}^k \lambda_{i_p}\right) t}$ will appear in the expression of $V_k^2(t)$, where $C$ is a constant depending on the eigenvalues and initial value, and $C>0$.

\par
By Proposition \ref{prop curvature}, the square of the $i$th curvature is
\begin{align}\label{align ki^2}
\kappa_i^2(t)
=\frac{ V_{i-1}^2(t) V_{i+1}^2(t) }{ V_{1}^2(t) V_{i}^4(t) } \quad
(i=1,2,\cdots,m-1).
\end{align}
Now, we consider the limit of $\kappa_i(t)$ as $t\to+\infty$ by comparing the exponents of $\mathrm{e}$ in the numerator and denominator of $\kappa_i^2(t)$.
Let $\Delta_1$ and $\Delta_2$ denote the maximum values of $\alpha$ in the terms of the form $\mathrm{e}^{\alpha t}$ in $ V_{i-1}^2(t) V_{i+1}^2(t)$ and $V_{1}^2(t) V_{i}^4(t)$, respectively.
We define
\begin{align*}
&\lambda_{(1)}=\max\left\{ \sigma(A) \backslash \left\{ 0 \right\} \right\},
\\&\lambda_{(2)}=\max\left\{ \sigma(A) \backslash \left\{ 0, \lambda_{(1)} \right\} \right\},
\\&\cdots\cdots,
\\&\lambda_{(i)}=\max\left\{ \sigma(A) \backslash \left\{ 0, \lambda_{(1)}, \lambda_{(2)}, \cdots, \lambda_{(i-1)} \right\} \right\},
\\&\cdots\cdots
\end{align*}
Then by (\ref{Vk^2}) and (\ref{align ki^2}), we have
\begin{align*}
\Delta_1=2\sum_{a=1}^{i-1}\lambda_{(a)}+2\sum_{b=1}^{i+1}\lambda_{(b)},
\quad \Delta_2=2\lambda_{(1)}+4\sum_{c=1}^{i}\lambda_{(c)}.
\end{align*}
Thus,
\begin{align*}
\Delta_1-\Delta_2=2\left(\lambda_{(i+1)}-\lambda_{(1)}-\lambda_{(i)}\right).
\end{align*}
It follows that
\begin{align}\label{align limit ki}
\lim\limits_{t\to+\infty}\kappa_i(t)=0 \iff \Delta_1<\Delta_2 \iff \lambda_{(1)}+\lambda_{(i)}>\lambda_{(i+1)}, \nonumber\\
\lim\limits_{t\to+\infty}\kappa_i(t)=C \iff \Delta_1=\Delta_2 \iff \lambda_{(1)}+\lambda_{(i)}=\lambda_{(i+1)}, \nonumber\\
\lim\limits_{t\to+\infty}\kappa_i(t)=+\infty \iff \Delta_1>\Delta_2 \iff \lambda_{(1)}+\lambda_{(i)}<\lambda_{(i+1)},
\end{align}
where $C$ is a positive constant depending on the initial value $r(0)=r_0 ~(r_{j}(0)\neq0$ for $j=1,2,\cdots,n)$.
Here we notice that for any given real diagonal matrix $A$,
if for a given initial value $r(0)\in\mathbb{R}^n$ that satisfies $\prod_{j=1}^n {r_{j}(0)}\neq0$,
we have $\lim\limits_{t\to+\infty}\kappa_i(t)=0$ (or $+\infty$, or a constant $C>0$, respectively),
then for an arbitrary $r(0)\in\mathbb{R}^n$ satisfying $\prod_{j=1}^n {r_{j}(0)}\neq0$,
we still have $\lim\limits_{t\to+\infty}\kappa_i(t)=0$ (or $+\infty$, or a constant $\tilde{C}>0$, respectively).

\par
Noting that $A$ is a real diagonal matrix, by Proposition \ref{prop asy.stable}, the zero solution of the system (\ref{system1}) is stable if and only if
$\lambda_i(A)\leqslant 0 \, (i=1,2,\cdots,n)$.
If the zero solution of the system is unstable, then we have $\lambda_{(1)}>0$, thus $\lambda_{(1)}+\lambda_{(i)}>\lambda_{(i+1)}$.
By (\ref{align limit ki}), we have $\lim\limits_{t\to+\infty}\kappa_i(t)=0$.
In other words, if
$\lim\limits_{t\to+\infty}\kappa_i(t)\neq0$ or $\lim\limits_{t\to+\infty}\kappa_i(t)$ does not exist,
then the zero solution of the system is stable.
\par
$(2)$ Suppose that $A$ is invertible, and $\lim\limits_{t\to+\infty}\kappa_i(t)\neq0$ or $\lim\limits_{t\to+\infty}\kappa_i(t)$ does not exist. Then $0$ is not a eigenvalue of $A$, and the zero solution of the system is stable.
By Proposition \ref{prop asy.stable}, the zero solution of the system is asymptotically stable.
\end{proof}
\par
Now, we proceed to the proof of Theorem \ref{thm main 1}.
\begin{proof}[Proof of Theorem \ref{thm main 1}]
Suppose that the linear time-invariant system $\dot{r}(t)=Ar(t)$ is equivalent to a system $\dot{v}(t)=Bv(t)$, where $B$ is a real diagonal matrix, $A=P^{-1}BP$, and $v(t)=Pr(t)$ is the equivalence transformation.
They by Proposition \ref{prop relation}, we have
\begin{align}\label{align 1}
\lim\limits_{t\to+\infty}\kappa_{r,i}(t)=0 \iff \lim\limits_{t\to+\infty}\kappa_{v,i}(t)=0.
\end{align}
\par
We define
\begin{align*}
\tilde S=\left\{ P^{-1}v(0) \Bigg| v(0)=(v_{1}(0),v_{2}(0),\cdots,v_{n}(0))^\mathrm{T}\in\mathbb{R}^n, \mathrm{s.t.}\ \prod_{i=1}^n {v_{i}(0)}\neq0 \right\}.
\end{align*}
Note that we can regard any given $n\times n$ invertible matrix $P$ as an invertible linear transformation $P: \mathbb{R}^n \to \mathbb{R}^n$, and the Lebesgue measure of $\mathbb{R}^n \backslash \tilde S$ satisfies
\begin{align}\label{align 2}
m\left( \mathbb{R}^n \backslash \tilde S \right)=0.
\end{align}
\par
If there exists a measurable set $E\subseteq \mathbb{R}^n$ whose Lebesgue measure is greater than $0$, such that $r(0)\in E$ implies that $\lim\limits_{t\to+\infty}\kappa_{r,i}(t)\neq0$ or $\lim\limits_{t\to+\infty}\kappa_{r,i}(t)$ does not exist, then by (\ref{align 1}) and (\ref{align 2}), there exists a $r(0)\in \tilde S$, such that the trajectory $v(t)$ with initial value $v(0)=Pr(0)$ satisfies $\lim\limits_{t\to+\infty}\kappa_{v,i}(t)\neq0$ or $\lim\limits_{t\to+\infty}\kappa_{v,i}(t)$ does not exist. Notice that when $r(0)\in \tilde S$, the vector $v(0)$ satisfies $\prod_{i=1}^n {v_{i}(0)}\neq0$, thus by Proposition \ref{prop diagonal}, the zero solution of the system $\dot{v}(t)=Bv(t)$ is stable, and then by Proposition \ref{prop similar}, the zero solution of the system $\dot{r}(t)=Ar(t)$ is also stable, which proves Theorem \ref{thm main 1} (1).
\par
Since $A$ is similar to $B$, the matrix $A$ is invertible if and only if $B$ is invertible.
The method of the proof of (1) works for (2), which completes the proof of Theorem \ref{thm main 1}.
\end{proof}

\section{Relationship Between Torsion and Stability}\label{Section Torsion}

In this section, we give the proof of Theorem \ref{thm main 2}, which establish a relationship between the torsion of the trajectory and the stability of the zero solution of the system. From Proposition \ref{prop similar}, Proposition \ref{prop relation}, and Proposition \ref{prop Jordan1},
we only need to focus on the case that $A$ is an invertible matrix in real Jordan canonical form (\ref{real Jordan matrix}), and prove the following result.

\begin{prop} \label{prop torsion Jordan}
Suppose that $\dot{r}(t)=Ar(t)$ is a linear time-invariant system, where $A$ is an $n\times n$ invertible matrix in real Jordan canonical form, and $r(t)\in\mathbb{R}^n$. Denote by $\tau(t)$ the torsion of trajectory of a solution $r(t)$. Then for any given initial value $r(0)\in S$,
we have
\par
$(1)$ if $\lim\limits_{t\to+\infty}\tau(t)\neq0$ or $\lim\limits_{t\to+\infty}\tau(t)$ does not exist, then the zero solution of the system is stable;
\par
$(2)$ if $\lim\limits_{t\to+\infty}\tau(t)=+\infty$, then the zero solution of the system is asymptotically stable.
\end{prop}

\subsection{Blocks $J_p(\lambda)$ and $C_m(a,b)$}%\label{subsection blocks}

\

In order to study the matrices in real Jordan canonical form (\ref{real Jordan matrix}), we first consider the blocks of the forms
\begin{align}\label{align J and C}
J_p(\lambda)=
\begin{pmatrix}
\lambda &1&&&\\[0.6em]
&\lambda &1&&\\
&&\lambda &\ddots&\\
&&&\ddots &1\\[0.5em]
&&&&\lambda
\end{pmatrix}_{p\times p}
\quad \text{and} \quad
C_m(a,b)=
\begin{pmatrix}
\Lambda & I_2&&&\\[0.6em]
&\Lambda & I_2&&\\
&&\Lambda &\ddots&\\
&&&\ddots & I_2\\[0.5em]
&&&&\Lambda
\end{pmatrix}_{2m\times 2m},
\end{align}
where
$\lambda,a,b\in\mathbb{R}$, $b>0$,
and
$
\Lambda=
\begin{pmatrix}
 a&b\\
-b&a
\end{pmatrix}
$. Part of this subsection goes back to the work as far as \cite{Wang2}.

\par
$(1)$ For a $J_p(\lambda)$ block, by direct calculation, we obtain
\begin{align}\label{align J^2}
J_p^2(\lambda)=
\begin{pmatrix}
\lambda^2 & 2\lambda & 1 && \\
& \lambda^2 & 2\lambda & \ddots & \\
&& \lambda^2 & \ddots & 1 \\
&&& \ddots & 2\lambda \\[0.5em]
&&&& \lambda^2
\end{pmatrix}_{p\times p},
\quad
J_p^3(\lambda)=
\begin{pmatrix}
\lambda^3 & 3\lambda^2 & 3\lambda & 1 && \\
& \lambda^3 & 3\lambda^2 & 3\lambda & \ddots & \\
&& \lambda^3 & 3\lambda^2 & \ddots & 1 \\
&&& \lambda^3 & \ddots & 3\lambda \\
&&&& \ddots & 3\lambda^2 \\[0.5em]
&&&&& \lambda^3
\end{pmatrix}_{p\times p},
\end{align}
and we have the exponential function
\begin{align}\label{align J exp}
\mathrm{e}^{tJ_p(\lambda)}=
\mathrm{e}^{\lambda t}
\begin{pmatrix}
1 & t & \frac{t^2}{2!} & \frac{t^3}{3!} & \cdots & \frac{t^{p-1}}{(p-1)!} \\[0.5em]
  & 1 & t & \frac{t^2}{2!} & \cdots & \frac{t^{p-2}}{(p-2)!} \\[0.5em]
  &   & 1 & t & \cdots & \frac{t^{p-3}}{(p-3)!} \\[0.5em]
  &   &   & \ddots & \ddots & \vdots \\[0.5em]
&&&& 1 & t \\[0.5em]
&&&&& 1
\end{pmatrix}.
\end{align}

\par
For the system $\dot{r}(t)=J_p(\lambda)r(t)$, by substituting (\ref{align J exp}) into $r(t)=\mathrm{e}^{tJ_p(\lambda)}r(0)$,
we obtain the expressions of the coordinates of $r(t)$
\begin{align}\label{align rk P}
 r_k(t)=\mathrm{e}^{\lambda t} P_{p;k}(t) \quad ( k=1,2,\cdots,p ),
\end{align}
where the polynomial
\begin{align}\label{align P}
P_{p;k}(t)=\sum_{l=0}^{p-k} \frac{r_{k+l}(0)}{l!}\, t^l.
\end{align}

\par
Substituting (\ref{align J and C}) and (\ref{align J^2}) into $r^{(s)}(t)=J_p^s(\lambda) r(t)$ for $s=1,2,3$, combined with (\ref{align rk P}), we see that the coordinates of the derivatives of $r(t)$ are
\begin{align}\label{align J der}
 \dot r_k(t)
 &=\lambda r_k(t) + r_{k+1}(t)
 =\mathrm{e}^{\lambda t} \left\{ \lambda P_{p;k}(t) + P_{p;k+1}(t) \right\}, \nonumber\\[0.5em]
 \ddot r_k(t)
 &=\lambda^2 r_k(t) + 2\lambda r_{k+1}(t) + r_{k+2}(t)
 =\mathrm{e}^{\lambda t} \left\{ \lambda^2 P_{p;k}(t) + 2\lambda P_{p;k+1}(t) + P_{p;k+2}(t) \right\}, \nonumber\\[0.5em]
  \dddot r_k(t)
 &=\lambda^3 r_k(t) +3\lambda^2 r_{k+1}(t) + 3\lambda r_{k+2}(t) + r_{k+3}(t) \nonumber\\
 &=\mathrm{e}^{\lambda t} \left\{ \lambda^3 P_{p;k}(t) +3\lambda^2 P_{p;k+1}(t) + 3\lambda P_{p;k+2}(t) + P_{p;k+3}(t) \right\},
\end{align}
where we have a convention that $r_k(t)=0$ for $k>p$.
\par
We see that if $k\in\{1,2,\cdots,p\}$, then  $\deg(P_{p;k}(t))=p-k$;
if $k>p$, then $P_{p;k}(t)=0$.

\par
$(2)$ For a $C_m(a,b)$ block, a direct calculation gives
\begin{align}\label{align C^2}
C_m^2(a,b)=
\begin{pmatrix}
\Lambda^2 & 2\Lambda & I_2 && \\
& \Lambda^2 & 2\Lambda & \ddots & \\
&& \Lambda^2 & \ddots & I_2 \\[-0.1em]
&&& \ddots & 2\Lambda \\[0.5em]
&&&& \Lambda^2
\end{pmatrix}_{2m\times 2m},
\quad
C_m^3(a,b)=
\begin{pmatrix}
\Lambda^3 & 3\Lambda^2 & 3\Lambda & I_2 && \\
& \Lambda^3 & 3\Lambda^2 & 3\Lambda & \ddots & \\
&& \Lambda^3 & 3\Lambda^2 & \ddots & I_2 \\
&&& \Lambda^3 & \ddots & 3\Lambda \\
&&&& \ddots & 3\Lambda^2 \\[0.5em]
&&&&& \Lambda^3
\end{pmatrix}_{2m\times 2m},
\end{align}
where
$
\Lambda^2=
\begin{pmatrix}
a^2-b^2 & 2ab \\[0.7em]
-2ab & a^2-b^2
\end{pmatrix}
$,
and
$
\Lambda^3=
\begin{pmatrix}
a(a^2-3b^2) & b(3a^2-b^2) \\[0.7em]
-b(3a^2-b^2) & a(a^2-3b^2)
\end{pmatrix}
$;
and we have the exponential function
\begin{align}\label{align C exp}
\mathrm{e}^{tC_m(a,b)}=
\mathrm{e}^{a t}
\begin{pmatrix}
R & tR & \frac{t^2}{2!}R & \frac{t^3}{3!}R & \cdots & \frac{t^{m-1}}{(m-1)!}R \\[0.7em]
  & R & tR & \frac{t^2}{2!}R & \cdots & \frac{t^{m-2}}{(m-2)!}R \\[0.7em]
  &   & R & tR & \cdots & \frac{t^{m-3}}{(m-3)!}R \\[0.7em]
  &   &   & \ddots & \ddots & \vdots \\[0.7em]
&&&& R & tR \\[0.7em]
&&&&& R
\end{pmatrix},
\end{align}
where
$
R=
\begin{pmatrix}
\cos bt & \sin bt \\[0.7em]
-\sin bt & \cos bt
\end{pmatrix}
$.

\par
For the system $\dot{r}(t)=C_m(a,b)r(t)$, write
\begin{align*}
r(t)&=(r_{1}(t),r_{2}(t),\cdots,r_{2m-1}(t),r_{2m}(t))^\mathrm{T} \\
&=(r_{1,1}(t),r_{1,2}(t),r_{2,1}(t),r_{2,2}(t),\cdots,r_{m,1}(t),r_{m,2}(t))^\mathrm{T}.
\end{align*}
Substituting (\ref{align C exp}) into $r(t)=\mathrm{e}^{tC_m(a,b)}r(0)$,
we obtain the expressions of the coordinates of $r(t)$
\begin{align}\label{align ri C}
r_{i,1}(t)=\mathrm{e}^{at}T_{m;i,1}(t),\quad
r_{i,2}(t)=\mathrm{e}^{at}T_{m;i,2}(t)\quad (i=1,2,\cdots,m),
\end{align}
where
\begin{align}\label{align T}
\left\{ \
\begin{aligned}
T_{m;i,1}(t)&=\sum_{k=0}^{m-i}\frac{t^k}{k!}(r_{2i+2k-1}(0)\cos{bt}+r_{2i+2k}(0)\sin{bt}), \\
T_{m;i,2}(t)&=\sum_{k=0}^{m-i}\frac{t^k}{k!}(-r_{2i+2k-1}(0)\sin{bt}+r_{2i+2k}(0)\cos{bt}).
\end{aligned}
\right.
\end{align}
By (\ref{align T}), we have
\begin{align}\label{align T^2+T^2}
T_{m;1,1}^2(t)+ T_{m;1,2}^2(t)
=\frac{ r_{m,1}^2(0)+r_{m,2}^2(0)}{\left[(m-1)!\right]^2} t^{2m-2} + \sum_{\varphi=0}^{2m-3} t^{\varphi} B_\varphi(t),
\end{align}
where each $B_\varphi(t)$ is a bounded function.

\par
Substituting (\ref{align J and C}) and (\ref{align C^2}) into $r^{(s)}(t)=C_m^s(a,b) r(t)$ for $s=1,2,3$,
combined with (\ref{align ri C}),
we see that the coordinates of the derivatives of $r(t)$ are

\begin{align}\label{align C der}
\dot r_{i,1}(t)
=&ar_{i,1}(t)+br_{i,2}(t)+r_{i+1,1}(t)
=\mathrm{e}^{at}\left\{aT_{m;i,1}(t)+bT_{m;i,2}(t)+T_{m;i+1,1}(t)\right\}, \nonumber\\[0.5em]
\dot r_{i,2}(t)
=&-br_{i,1}(t)+ar_{i,2}(t)+r_{i+1,2}(t)
=\mathrm{e}^{at}\left\{-bT_{m;i,1}(t)+aT_{m;i,2}(t)+T_{m;i+1,2}(t)\right\}, \nonumber\displaybreak[0]\\[0.5em]
\ddot r_{i,1}(t)
=& \left(a^2-b^2\right)r_{i,1}(t)+2abr_{i,2}(t)+2ar_{i+1,1}(t)+2br_{i+1,2}(t)+r_{i+2,1}(t) \nonumber\\
=& \mathrm{e}^{at} \left\{\left(a^2-b^2\right)T_{m;i,1}(t)+2abT_{m;i,2}(t)+2aT_{m;i+1,1}(t)+2bT_{m;i+1,2}(t)+T_{m;i+2,1}(t) \right\}, \nonumber\\[0.5em]
\ddot r_{i,2}(t)
=& -2abr_{i,1}(t)+\left(a^2-b^2\right)r_{i,2}(t)-2br_{i+1,1}(t)+2ar_{i+1,2}(t)+r_{i+2,2}(t) \nonumber\\
=& \mathrm{e}^{at} \left\{-2abT_{m;i,1}(t)+\left(a^2-b^2\right)T_{m;i,2}(t)-2bT_{m;i+1,1}(t)+2aT_{m;i+1,2}(t)+T_{m;i+2,2}(t) \right\}, \nonumber\displaybreak[0]\\[0.5em]
\dddot r_{i,1}(t)
=&a\left(a^2-3b^2\right)r_{i,1}(t)+b\left(3a^2-b^2\right)r_{i,2}(t)
+3\left(a^2-b^2\right)r_{i+1,1}(t)+6abr_{i+1,2}(t) \nonumber\\
&+3ar_{i+2,1}(t)+3br_{i+2,2}(t)
+r_{i+3,1}(t) \nonumber\\
=&\mathrm{e}^{at}\left\{a\left(a^2-3b^2\right)T_{m;i,1}(t)+b\left(3a^2-b^2\right)T_{m;i,2}(t)
+3\left(a^2-b^2\right)T_{m;i+1,1}(t)+6abT_{m;i+1,2}(t) \right. \nonumber\\
&\left.+3aT_{m;i+2,1}(t)+3bT_{m;i+2,2}(t)
+T_{m;i+3,1}(t)\right\}, \nonumber\\[0.5em]
\dddot r_{i,2}(t)
=&-b\left(3a^2-b^2\right)r_{i,1}(t)+a\left(a^2-3b^2\right)r_{i,2}(t)
-6abr_{i+1,1}(t)+3\left(a^2-b^2\right)r_{i+1,2}(t) \nonumber\\
&-3br_{i+2,1}(t)+3ar_{i+2,2}(t)
+r_{i+3,2}(t) \nonumber\\
=&\mathrm{e}^{at}\left\{-b\left(3a^2-b^2\right)T_{m;i,1}(t)+a\left(a^2-3b^2\right)T_{m;i,2}(t)
-6abT_{m;i+1,1}(t)+3\left(a^2-b^2\right)T_{m;i+1,2}(t) \right. \nonumber\\
&\left.-3bT_{m;i+2,1}(t)+3aT_{m;i+2,2}(t)
+T_{m;i+3,2}(t)\right\}, \nonumber\\
\end{align}
where we have a convention that if $i>m$, then
$r_{i,j}(t)=0 ~(j=1,2)$.

\par
It should be noted that in the following subsections we will consider the case where $A$ has more than one block of the form $J_p(\lambda)$ or $C_m(a,b)$, so when $P_{p;k}(t)$, $T_{m;i,1}(t)$ and $T_{m;i,2}(t)$ appear in the following, the $r_{k+l}(0)$ in (\ref{align P}) should be understood as the coordinate of $r(t)$ which corresponds to the $(k+l)$th row of the diagonal block corresponding to the $P_{p;k}(t)$, and the $r_{2i+2k-1}(0)$ and $r_{2i+2k}(0)$ in (\ref{align T}) should be understood as the coordinates of $r(t)$ which correspond to the $(2i+2k-1)$th and $(2i+2k)$th row of the diagonal block corresponding to the $T_{m;i,1}(t)$ and $T_{m;i,2}(t)$, respectively.

\subsection{Function $V_2(t)$}\label{subsection V2}

\

By Proposition \ref{prop V}, we have
\begin{align}\label{align V2}
V_2^2(t)
=
\sum_{1\leqslant i<j\leqslant n}
\begin{vmatrix}
\dot{r}_{i}(t) & \ddot{r}_{i}(t) \\[1em]
\dot{r}_{j}(t) & \ddot{r}_{j}(t)
\end{vmatrix}^2.
\end{align}
Considering the form of the expression of torsion $\tau(t)$, it is necessary to make a detailed analysis of the function $V_2(t)$.

\begin{lem} \label{lemma V2}
Suppose that $\dot{r}(t)=Ar(t)$ is a linear time-invariant system, where $A$ is an $n\times n$ matrix in real Jordan canonical form, and $r(t)\in\mathbb{R}^n$. The function $V_2(t)$ is given by (\ref{align V2}). Then for any given $r(0)\in S$, we have
\par
$(1)$ $V_2(t)\equiv 0$ if and only if
\begin{align}\label{align V2=0}
A=
\begin{pmatrix}
\lambda &&& \\[0.2em]
&\ddots && \\[0.2em]
&&\lambda & \\[0.2em]
&&& 0_{z\times z}
\end{pmatrix}
\quad (\lambda\in\mathbb{R})
\quad \text{or}
\quad
A=
\begin{pmatrix}
J_2(0) &&& \\[0.5em]
&\ddots && \\[0.5em]
&& J_2(0) & \\[0.5em]
&&& 0_{z\times z}
\end{pmatrix},
\end{align}
where $z\in\{0,1,\cdots,n\}$;
\par
$(2)$ if $V_2(t)\not\equiv 0$, then there exists a $T>0$, such that $V_2(t)>0$ for all $t>T$.
\end{lem}

\begin{proof}
Suppose $A$ is an $n\times n$ matrix in real Jordan canonical form.
\par
(a) If $A$ has a diagonal block $C_m(a,b)$ (without loss of generality, we assume that this $C_m(a,b)$ block is the first diagonal block of $A$), then by (\ref{align ri C}), (\ref{align T^2+T^2}), (\ref{align C der}), and the analysis of Subsection $4.4$ of \cite{Wang2}, we have
\begin{align}\label{align C V_2^2}
V_2^2(t)
\geqslant
\sum_{1\leqslant i<j\leqslant 2m}
\begin{vmatrix}
\dot{r}_{i}(t) & \ddot{r}_{i}(t) \\[1em]
\dot{r}_{j}(t) & \ddot{r}_{j}(t)
\end{vmatrix}^2
=
\mathrm{e}^{4at} \left( C t^{4m-4} + \sum_{\varphi=0}^{4m-5} t^{\varphi} B_\varphi(t) \right),
\end{align}
where the constant
$
C=\frac{ b^2 \left(a^2+b^2\right)^2 \left(r_{m,1}^2(0)+r_{m,2}^2(0)\right)^2 }{\left[(m-1)!\right]^4}>0,
$
and
$
B_\varphi(t) ~( \varphi=0,1,\cdots,4m-5 )
$
are bounded functions.
It follows that there exists a $T>0$, such that $V_2^2(t)>0$ for all $t>T$.

\par
(b) If A has a diagonal block $J_p(\lambda)$, where $p\geqslant 3$ or
$\left\{
\begin{aligned}
&p=2, \\
&\lambda\neq 0,
\end{aligned}\right.$
then by (\ref{align rk P}), (\ref{align J der}), and the analysis of Subsection $4.2$ of \cite{Wang2}, we have
\begin{align*}%\label{align RH V_2^2}
\sum_{1\leqslant i<j\leqslant p}
\begin{vmatrix}
\dot{r}_{i}(t) & \ddot{r}_{i}(t) \\[1em]
\dot{r}_{j}(t) & \ddot{r}_{j}(t)
\end{vmatrix}^2
=
\mathrm{e}^{4\lambda t} f(t),
\end{align*}
where $f(t)$ is a polynomial, and
\par
(b1) if $p\geqslant 3$, then
$\deg(f(t))=\left\{
\begin{aligned}
&4(p-2), &\lambda\neq0,\\
&4(p-3), &\lambda=0;
\end{aligned}\right.$
\par
(b2) if $p=2$ and $\lambda\neq 0$, then
$f(t)=\lambda^4 r_{2}^4(0)>0$. \\
we see that for both (b1) and (b2), there exists a $T>0$, such that $f(t)>0$ for all $t>T$,
thus
\begin{align}\label{align J V_2^2}
V_2^2(t)\geqslant \sum_{1\leqslant i<j\leqslant p}
\begin{vmatrix}
\dot{r}_{i}(t) & \ddot{r}_{i}(t) \\[1em]
\dot{r}_{j}(t) & \ddot{r}_{j}(t)
\end{vmatrix}^2
=
\mathrm{e}^{4\lambda t} f(t)
>0
\end{align}
for all $t>T$.

\par
(c) If A has $J_1(\lambda_1)$ and $J_1(\lambda_2)$ as its diagonal blocks,
where $\lambda_1\neq \lambda_2$ and $\lambda_1 \lambda_2\neq 0$, without loss of generality we can assume $A=\mathrm{diag}\{ J_1(\lambda_1), J_1(\lambda_2), \cdots \}$, then by (\ref{Vk^2}), we have
\begin{align*}
V_2^2(t)\geqslant
\mathrm{e}^{2( \lambda_1 + \lambda_2 ) t}
\left\{ \lambda_1 \lambda_2 (\lambda_2 - \lambda_1)r_{1}(0)r_{2}(0) \right\}^2
>0.
\end{align*}

\par
(d) If both $J_2(0)$ and $J_1(\lambda)\ (\lambda\neq 0)$ are diagonal blocks of $A$, without loss of generality we can assume $A=\mathrm{diag}\{ J_2(0), J_1(\lambda), \cdots \}$, then we have
\begin{align*}
V_2^2(t)
&\geqslant
\begin{vmatrix}
\dot{r}_{1}(t) & \ddot{r}_{1}(t) \\[1em]
\dot{r}_{3}(t) & \ddot{r}_{3}(t)
\end{vmatrix}^2
=
\begin{vmatrix}
r_2(t) & 0 \\[1em]
\lambda r_3(t) & \lambda^2 r_3(t)
\end{vmatrix}^2
\\&=
\begin{vmatrix}
r_2(0) & 0 \\[1em]
\lambda \mathrm{e}^{\lambda t} r_3(0) & \lambda^2 \mathrm{e}^{\lambda t} r_3(0)
\end{vmatrix}^2
=
\left( \lambda^2 \mathrm{e}^{\lambda t} r_2(0) r_3(0) \right)^2
>0.
\end{align*}

\par
In the case of (a)(b)(c)(d), we have show that there exists a $T>0$, such that $V_2(t)>0$ for all $t>T$.
Note that (a)(b)(c)(d) cover all cases where $A$ is a matrix in real Jordan canonical form except the two cases in (\ref{align V2=0}).
Nevertheless, by direct calculation, we have $V_2(t)\equiv 0$ for the two cases in (\ref{align V2=0}), which completes the proof.
\end{proof}

\par
From Lemma \ref{lemma V2}, we know that except for the two trivial cases in (\ref{align V2=0}), we have $V_2(t)>0$ when $t$ is sufficiently large, that is to say, there exists a $T>0$, such that we have the expression (\ref{align torsion}) of torsion $\tau(t)$ for all $t>T$, which avoids a lot of potential trouble when we consider the limit of $\tau(t)$ as $t\to +\infty$ in the proof of Theorem \ref{thm main 2}.

\subsection{Function $V_3(t)$}%\label{subsection V3}

\

The function $V_3(t)$ is given by Proposition \ref{prop V}. In fact, we have
\begin{align}\label{align V3}
V_3^2(t)=
\sum_{1\leqslant i<j<k\leqslant n}
\begin{vmatrix}
\dot{r}_{i}(t) & \ddot{r}_{i}(t) & \dddot{r}_{i}(t) \\[1em]
\dot{r}_{j}(t) & \ddot{r}_{j}(t) & \dddot{r}_{j}(t) \\[1em]
\dot{r}_{k}(t) & \ddot{r}_{k}(t) & \dddot{r}_{k}(t)
\end{vmatrix}^2.
\end{align}
\par
By (\ref{align J der}) and (\ref{align C der}), we see that all coordinates of $r^{(s)}(t)\,(s=1,2,3)$ can be expressed in the form of
\begin{align*}
\dot r_{i;k}(t)&=\mathrm{e}^{\mathrm{Re}(\lambda_i)t} f_{i;k}(t), \\
\ddot r_{i;k}(t)&=\mathrm{e}^{\mathrm{Re}(\lambda_i)t} g_{i;k}(t), \\
\dddot r_{i;k}(t)&=\mathrm{e}^{\mathrm{Re}(\lambda_i)t} h_{i;k}(t),
\end{align*}
where $r_{i;k}^{(s)}(t)$ denotes the coordinate of $r^{(s)}(t)$ corresponding to the $k$th row of the $i$th diagonal block of $A$. Hence

\begin{align}\label{align V3 det}
\begin{vmatrix}
\dot{r}_{i_1;k_1}(t) & \ddot{r}_{i_1;k_1}(t) & \dddot{r}_{i_1;k_1}(t) \\[1em]
\dot{r}_{i_2;k_2}(t) & \ddot{r}_{i_2;k_2}(t) & \dddot{r}_{i_2;k_2}(t) \\[1em]
\dot{r}_{i_3;k_3}(t) & \ddot{r}_{i_3;k_3}(t) & \dddot{r}_{i_3;k_3}(t)
\end{vmatrix}^2
=\mathrm{e}^{2\left\{\mathrm{Re}\left(\lambda_{i_1}\right)+\mathrm{Re}\left(\lambda_{i_2}\right)+\mathrm{Re}\left(\lambda_{i_3}\right)\right\}t} G(t),
\end{align}
where
$G(t)$
is a linear combination of terms in the form of $t^\beta B_\gamma(t)$,
where $B_\gamma(t)$ is a bounded function.

\par
In the remainder of this paper, set
\begin{align*}%\label{align M}
M=\max\{\mathrm{Re}(\lambda)| \lambda\in \sigma(A)\}.
\end{align*}
Then by (\ref{align V3}) and (\ref{align V3 det}), we obtain
\begin{align}\label{align eta}
\Delta_1\leqslant 6M,
\end{align}
where $\Delta_1$ denotes the maximum values of $\alpha$ in the terms of the form $\mathrm{e}^{\alpha t} t^\beta B_\gamma(t)$ in $V_3^2(t)$.
%In the remainder of this paper, let $\eta$ and $\theta$ denote the maximum values of $\mu$ in the terms of the form $\mathrm{e}^{\mu t} t^\nu B_\omega(t)$ in $V_3^2(t)$ and $V_{2}^4(t)$, respectively.

\subsection{Proof of Theorem \ref{thm main 2} (1)}

\

In order to give a proof of Theorem \ref{thm main 2} (1), we only need to prove Proposition \ref{prop torsion Jordan} (1).
In this subsection, we will discuss the two cases in which the zero solution of the system is unstable, and obtain $\lim\limits_{t\to+\infty}\tau(t)=0$.
In fact, we will prove Lemma \ref{lemma M>0} and Lemma \ref{lemma M=0 CH}.

\begin{lem} \label{lemma M>0}
Under the assumptions of Proposition \ref{prop torsion Jordan}, if $M>0$, then for any given $r(0)\in S$, we have $\lim\limits_{t\to+\infty}\tau(t)=0$.
\end{lem}

\begin{proof}
Suppose $M>0$. Note that
\begin{align}\label{align tor^2}
\tau^2(t)
=\frac{ V_{3}^2(t) }{ V_{2}^4(t) }
=\frac{
 \sum_{1\leqslant i<j<k\leqslant n}
\begin{vmatrix}
\dot{r}_{i}(t) & \ddot{r}_{i}(t) & \dddot{r}_{i}(t) \\[1em]
\dot{r}_{j}(t) & \ddot{r}_{j}(t) & \dddot{r}_{j}(t) \\[1em]
\dot{r}_{k}(t) & \ddot{r}_{k}(t) & \dddot{r}_{k}(t)
\end{vmatrix}^2
}
{
\left(
 \sum_{1\leqslant p<q\leqslant n}
\begin{vmatrix}
\dot{r}_{p}(t) & \ddot{r}_{p}(t) \\[1em]
\dot{r}_{q}(t) & \ddot{r}_{q}(t)
\end{vmatrix}^2
\right)^2
},
\end{align}
where the functions $V_{3}^2(t)$ and $V_{2}^4(t)$ are both linear combinations of terms in the form of $\mathrm{e}^{\alpha t} t^\beta B_\gamma(t)$,
where each $B_\gamma(t)$ is a bounded function.
We will prove $\lim\limits_{t\to+\infty}\tau(t)=0$ for the following cases. For simplicity, let $t>0$.

\par
(a) If $A$ has a diagonal block $C_m(M,b)$,
then by
(\ref{align C V_2^2}), (\ref{align V3 det}), and (\ref{align eta}), we have
\begin{align*}
0
\leqslant \tau^2(t)
=\frac{ V_{3}^2(t) }{ V_{2}^4(t) }
&\leqslant \frac{ \mathrm{e}^{6Mt} F(t) + R(t) }{ \left\{\mathrm{e}^{4Mt} \left( C t^{4m-4} + \sum_{\varphi=0}^{4m-5} t^{\varphi} B_\varphi(t) \right) \right\}^2 } \\
&= \frac{ \mathrm{e}^{6Mt} F(t) + R(t) }{ \mathrm{e}^{8Mt} \left( C^2 t^{8m-8} + \sum_{\psi=0}^{8m-9} t^{\psi} B_\psi(t) \right) }
\to 0
\quad
(t\to +\infty),
\end{align*}
where the constant $C>0$, all $B_\varphi(t)$ and $B_\psi(t)$ are bounded functions, the function $F(t)$ is a linear combination of terms in the form of $t^\beta B_\gamma(t)$, and $R(t)$ is a linear combination of terms in the form of $\mathrm{e}^{\alpha t} t^\beta B_\gamma(t)$, where $\alpha<6M$, and each $B_\gamma(t)$ is a bounded function.
Hence we obtain
$\lim\limits_{t\to+\infty}\tau(t)=0$.

\par
(b) If $A$ has a diagonal block $J_p(M)\, (p\geqslant 2)$,
then by
(\ref{align J V_2^2}), (\ref{align V3 det}), and (\ref{align eta}), we have
\begin{align*}
0
\leqslant \tau^2(t)
=\frac{ V_{3}^2(t) }{ V_{2}^4(t) }
& \leqslant \frac{ \mathrm{e}^{6Mt} F(t) + R(t) }{ \mathrm{e}^{8Mt} f^2(t) }
\to 0
\quad
(t\to +\infty),
\end{align*}
where $f(t)$ is a polynomial satisfying $f(t)>0$ and $\deg(f(t))=4(p-2)$,
the function $F(t)$ is a linear combination of terms in the form of $t^\beta B_\gamma(t)$,
and $R(t)$ is a linear combination of terms in the form of $\mathrm{e}^{\alpha t} t^\beta B_\gamma(t)$, where $\alpha<6M$, and each $B_\gamma(t)$ is a bounded function. Hence we obtain
$\lim\limits_{t\to+\infty}\tau(t)=0$.

\par
(c) If in $A$ only those $J_1(M)$ blocks are diagonal blocks satisfying $\mathrm{Re}(\lambda)=M$, then we should consider the eigenvalues whose real part is less than $M$.
In fact, suppose two $J_1(M)$ diagonal blocks are in the $i$th and $j$th row of $A$, respectively. Then
\begin{align*}
\begin{vmatrix}
\dot{r}_{i}(t) & \ddot{r}_{i}(t) \\[1em]
\dot{r}_{j}(t) & \ddot{r}_{j}(t)
\end{vmatrix}^2
=\begin{vmatrix}
\mathrm{e}^{Mt} M r_i(0) & \mathrm{e}^{Mt} M^2 r_i(0) \\[1em]
\mathrm{e}^{Mt} M r_j(0) & \mathrm{e}^{Mt} M^2 r_j(0)
\end{vmatrix}^2
=0,
\end{align*}
which means this term has no contribution to the value of $V_2^2(t)$.
In addition, note that $J_1(0)$ diagonal blocks in $A$ do not affect the value of $\tau(t)$.
We define
\begin{align*}%\label{align M}
N=\max\{\mathrm{Re}(\lambda)| \lambda\in \tilde\sigma(A)\backslash \{M\} \},
\end{align*}
where $\tilde\sigma(A)$ denotes the set of eigenvalues of $A$ which excluding the zero eigenvalues in $J_1(0)$ blocks.

\par
(c1) Suppose that $A$ has a diagonal block $C_m(N,b)$.
Let $r_M(t)$ denote the coordinate of $r(t)$ corresponding to the row of a diagonal block $J_1(M)$ of $A$,
and $r_{N,1}(t)$, $r_{N,2}(t)$ denote the coordinate of $r(t)$ corresponding to the first and second row of the diagonal block $C_m(N,b)$ of $A$, respectively.
Then by (\ref{align T^2+T^2}) and (\ref{align C der}), we have
\begin{align}\label{align N C}
&\begin{vmatrix}
\dot{r}_{M}(t) & \ddot{r}_{M}(t) \\[1em]
\dot{r}_{N,1}(t) & \ddot{r}_{N,1}(t)
\end{vmatrix}^2
+\begin{vmatrix}
\dot{r}_{M}(t) & \ddot{r}_{M}(t) \\[1em]
\dot{r}_{N,2}(t) & \ddot{r}_{N,2}(t)
\end{vmatrix}^2
\nonumber\\
=&\begin{vmatrix}
\mathrm{e}^{Mt} M r_M(0) & \mathrm{e}^{Mt} M^2 r_M(0) \\[1em]
\mathrm{e}^{Nt}\left\{NT_{m;1,1}(t)+bT_{m;1,2}(t)+T_{m;2,1}(t)\right\} &\quad \mathrm{e}^{Nt} \left\{\left(N^2-b^2\right)T_{m;1,1}(t)+2NbT_{m;1,2}(t)+\cdots \right\}
\end{vmatrix}^2
\nonumber\\
&+\begin{vmatrix}
\mathrm{e}^{Mt} M r_M(0) & \mathrm{e}^{Mt} M^2 r_M(0) \\[1em]
\mathrm{e}^{Nt}\left\{-bT_{m;1,1}(t)+NT_{m;1,2}(t)+T_{m;2,2}(t)\right\} &\quad \mathrm{e}^{Nt} \left\{-2NbT_{m;1,1}(t)+\left(N^2-b^2\right)T_{m;1,2}(t)+\cdots \right\}
\end{vmatrix}^2
\nonumber\displaybreak[0]\\
=&\mathrm{e}^{2(M+N)t} M^2 r_M^2(0)
\left\{
\left( N^2+b^2 \right) \left[ (M-N)^2+b^2 \right] \left[ T_{m;1,1}^2(t)+ T_{m;1,2}^2(t) \right] + \sum_{\chi=0}^{2m-3} t^{\chi} B_\chi(t)
\right\} \nonumber\\
=&\mathrm{e}^{2(M+N)t} M^2 r_M^2(0)
\left\{
\left( N^2+b^2 \right) \left[ (M-N)^2+b^2 \right] \left[ \frac{ r_{m,1}^2(0)+r_{m,2}^2(0)}{\left[(m-1)!\right]^2} t^{2m-2} + \sum_{\varphi=0}^{2m-3} t^{\varphi} B_\varphi(t) \right] + \sum_{\chi=0}^{2m-3} t^{\chi} B_\chi(t)
\right\} \nonumber\\
=&\mathrm{e}^{2(M+N)t} \left( C t^{2m-2} + \sum_{\psi=0}^{2m-3} t^{\psi} B_\psi(t) \right), \nonumber\\
\end{align}
where the constant $C= M^2 \left( N^2+b^2 \right) \left\{ (M-N)^2+b^2 \right\} r_M^2(0) \frac{ r_{m,1}^2(0)+r_{m,2}^2(0)}{\left[(m-1)!\right]^2}>0$, and all $B_\chi(t)$, $B_\varphi(t)$, and $B_\psi(t)$ are bounded functions.

\par
(c2) Suppose that $A$ has a diagonal block $J_p(N)$.
Let $r_M(t)$ denote the coordinate of $r(t)$ corresponding to the row of a diagonal block $J_1(M)$ of $A$,
and $r_N(t)$ the coordinate of $r(t)$ corresponding to the first row of the diagonal block $J_p(N)$ of $A$.
Then by (\ref{align P}) and (\ref{align J der}), we have
\begin{align}\label{align N P}
\begin{vmatrix}
\dot{r}_{M}(t) & \ddot{r}_{M}(t) \\[1em]
\dot{r}_{N}(t) & \ddot{r}_{N}(t)
\end{vmatrix}^2
=&\begin{vmatrix}
\mathrm{e}^{Mt} M r_M(0) & \mathrm{e}^{Mt} M^2 r_M(0) \\[1em]
\mathrm{e}^{Nt} \left\{ N P_{p;1}(t) + P_{p;2}(t) \right\}
 &\quad \mathrm{e}^{Nt} \left\{ N^2 P_{p;1}(t) + 2N P_{p;2}(t) + P_{p;3}(t) \right\}
\end{vmatrix}^2 \nonumber\\
=&
\left\{
\begin{aligned}
&\mathrm{e}^{2(M+N)t} \left( C t^{2p-2} + \sum_{\varphi=0}^{2p-3} t^{\varphi} B_\varphi(t) \right), &N\neq0,\\
&\mathrm{e}^{2(M+N)t} \left( \hat C t^{2p-4} + \sum_{\psi=0}^{2p-5} t^{\psi} B_\psi(t) \right), &N=0\ \text{and}\ p\geqslant2,
\end{aligned}\right.
\end{align}
where the constants $C, \hat C>0$, and all $B_\varphi(t)$ and $B_\psi(t)$ are bounded functions.

\par
By (c1) and (c2), we can give the expression of $V_2^2(t)$ in case (c).
In fact, we suppose
\begin{align*}
&C_{m_1}(N,b_1), C_{m_2}(N,b_2), \cdots, C_{m_k}(N,b_k),
J_{p_1}(N), J_{p_2}(N), \cdots, J_{p_l}(N) \\
&(m_1 \geqslant m_2 \geqslant\cdots\geqslant m_k, ~ \text{and} ~  p_1 \geqslant p_2 \geqslant\cdots\geqslant p_l)
\end{align*}
are the all diagonal blocks whose eigenvalues satisfy $\mathrm{Re}(\lambda)=N$.
Then by (\ref{align V2}), (\ref{align N C}), and (\ref{align N P}), we obtain
\begin{align}\label{align (c) V_2^2}
V_2^2(t)
\geqslant\mathrm{e}^{2(M+N)t} \left( C t^{\nu} + \sum_{\varphi=0}^{\nu-1} t^{\varphi} B_\varphi(t) \right),
\end{align}
where the constant $C>0$,
\begin{align}
\nu=
\left\{
\begin{aligned}
&\max\{ 2m_1-2, ~ 2p_1-2 \}, &N\neq0,\\
&\max\{ 2m_1-2, ~ 2p_1-4 \}, &N=0,
\end{aligned}\right.
\end{align}
and each $B_\varphi(t)$ is a bounded function.
\par
In what follows, $\Delta_1$ and $\Delta_2$ denote the maximum values of $\alpha$ in the terms of the form $\mathrm{e}^{\alpha t} t^\beta B_\gamma(t)$ in $V_3^2(t)$ and $V_2^4(t)$, respectively. Then by (\ref{align (c) V_2^2}), we have
\begin{align*}
\Delta_2=4(M+N).
\end{align*}
In the determinant of (\ref{align V3 det}), we can see that at most one row corresponds to a diagonal block with eigenvalue $M$, and the real parts of eigenvalues of the diagonal blocks corresponding to the other two rows are not greater than $N$, otherwise the determinant vanishes in $V_3^2(t)$. Hence we have
\begin{align*}
\Delta_1 \leqslant 2(M+2N).
\end{align*}
Thus, we have $\Delta_1-\Delta_2 \leqslant -2M < 0$.
It follows that
\begin{align*}
0
\leqslant \tau^2(t)
=\frac{ V_{3}^2(t) }{ V_{2}^4(t) }
& \leqslant \frac{ \mathrm{e}^{\Delta_1 t} F(t) + R(t) }{ \mathrm{e}^{\Delta_2 t} \left( \tilde C t^{2\nu} + \sum_{\psi=0}^{2\nu-1} t^{\psi} B_\psi(t) \right) }
\to 0
\quad
(t\to +\infty),
\end{align*}
where the constant $\tilde C>0$, each $B_\psi(t)$ is a bounded function, the function $F(t)$ is a linear combination of terms in the form of $t^\beta B_\gamma(t)$, and $R(t)$ is a linear combination of terms in the form of $\mathrm{e}^{\alpha t} t^\beta B_\gamma(t)$, where $\alpha<\Delta_1$,
and each $B_\gamma(t)$ is a bounded function. Hence we obtain
$\lim\limits_{t\to+\infty}\tau(t)=0$.

\par
Note that (a)(b)(c) cover all cases that satisfy $M>0$, which completes the proof.
\end{proof}

\par
Now we give Lemma \ref{lemma M=0 CH}.

\begin{lem} \label{lemma M=0 CH}
Under the assumptions of Proposition \ref{prop torsion Jordan}, if $M=0$, and $A$ has a diagonal block $C_m(0,b)\, (m\geqslant 2)$, then for any given $r(0)\in S$, we have $\lim\limits_{t\to+\infty}\tau(t)=0$.
\end{lem}

\begin{proof}
Suppose $M=0$, and $A$ has a diagonal block $C_m(0,b)\, (m\geqslant 2)$. Then from (\ref{align eta}), we have
$\Delta_1\leqslant 0$. From (\ref{align V2}) and (\ref{align C V_2^2}), we have $\Delta_2=0$.
\par
If $\Delta_1<\Delta_2=0$, then we have $\lim\limits_{t\to+\infty}\tau(t)=0$.
\par
If $\Delta_1=\Delta_2=0$, in order to obtain the limit of $\tau(t)$ as $t\to+\infty$, we need to compare the highest power of $t$ of terms in the form $\mathrm{e}^{0t} t^\beta B_\gamma(t)$ in the numerator and denominator of $\tau^2(t)$.
Let $\Gamma_1$ and $\Gamma_2$ denote the maximum value of $\beta$ in the terms of the form $\mathrm{e}^{0t} t^\beta B_\gamma(t)$ in $V_3^2(t)$ and $V_2^4(t)$, respectively.
Then we have
\begin{align}\label{align Gamma1}
\Gamma_1\leqslant 6(m-1).
\end{align}
In fact, by (\ref{align T}) and (\ref{align C der}), for a diagonal block $C_m(0,b)\, (m\geqslant 2)$, the functions $T_{m;1,1}(t)$ and $T_{m;1,2}(t)$ can reach the highest power $m-1$ of $t$, namely $t^{m-1}$, thus $r_{1,1}^{(s)}(t)$ and $r_{1,2}^{(s)}(t)\ (s=1,2,3)$ corresponding the first two rows of $C_m(0,b)\, (m\geqslant 2)$ can reach the highest power $m-1$ of $t$. Hence by (\ref{align V3}) and (\ref{align V3 det}), we obtain (\ref{align Gamma1}).
In addition, by (\ref{align V2}) and (\ref{align C V_2^2}), we have
\begin{align*}
\Gamma_2=2(4m-4)=8(m-1).
\end{align*}
Therefore $\Gamma_1\leqslant 6(m-1)<8(m-1)=\Gamma_2$.
It follows that
\begin{align*}
0
\leqslant \tau^2(t)
=\frac{ V_{3}^2(t) }{ V_{2}^4(t) }
& \leqslant \frac{ \sum_{\varphi=0}^{\Gamma_1} t^{\varphi} B_\varphi(t) + R(t) }{ C t^{\Gamma_2} + \sum_{\psi=0}^{\Gamma_2-1} t^{\psi} B_\psi(t) }
\to 0
\quad
(t\to +\infty),
\end{align*}
where the constants $C>0$, all $B_\varphi(t)$ and $B_\psi(t)$ are bounded functions, and $R(t)$ is a linear combination of terms in the form of $\mathrm{e}^{\alpha t} t^\beta B_\gamma(t)$, where $\alpha<0$,
and each $B_\gamma(t)$ is a bounded function. Hence we obtain
$\lim\limits_{t\to+\infty}\tau(t)=0$.
\end{proof}

\par
Lemma \ref{lemma M>0} and Lemma \ref{lemma M=0 CH} show that under the assumptions of Proposition \ref{prop torsion Jordan}, if the zero solution of the system is unstable, then $\lim\limits_{t\to+\infty}\tau(t)=0$. That is to say, Proposition \ref{prop torsion Jordan} (1) is proved. Thus we proved Theorem \ref{thm main 2} (1).

\subsection{Proof of Theorem \ref{thm main 2} (2)}

\

We have proved Proposition \ref{prop torsion Jordan} (1), and in order to prove Proposition \ref{prop torsion Jordan} (2), we only need to prove the following lemma.

\begin{lem} \label{lemma M=0 C2}
Under the assumptions of Proposition \ref{prop torsion Jordan}, if $M=0$, and in matrix $A$ only those $C_1(0,b)$ blocks are diagonal blocks satisfying $\mathrm{Re}(\lambda)=0$, then for any given $r(0)\in S$, we have $\lim\limits_{t\to+\infty}\tau(t)=0$ or $\lim\limits_{t\to+\infty}\tau(t)=C$, where the constant $C>0$.
\end{lem}

\begin{proof}
Set
\begin{align*}
A=
\begin{pmatrix}
C_1(0,b_1) &&& \\[0.2em]
&C_1(0,b_2) &&& \\[0.2em]
&&\ddots && \\[0.2em]
&&&C_1(0,b_s) & \\[0.2em]
&&&& \tilde A
\end{pmatrix},
\end{align*}
where all eigenvalues of $\tilde A$ have negative real parts.
\par
(1) If $s=1$, then by (\ref{align V2}) and (\ref{align C V_2^2}), we have $\Delta_2=0$.
In the determinant of (\ref{align V3 det}), we can see that at most two rows correspond to the diagonal block $C_1(0,b_1)$, and the real part of eigenvalue of the diagonal block corresponding to the other row is negative.
Hence $\Delta_1<0$.
It follows that $\lim\limits_{t\to+\infty}\tau(t)=0$.
\par
(2) If $s>1$, then $\Delta_1\leqslant 0=\Delta_2$. By direct calculation, we have
\begin{align*}
\lim\limits_{t\to+\infty}\tau^2(t)
&=
\frac{
\sum_{1\leqslant i<j\leqslant s} b_i^2 b_j^2 \left( b_i^2-b_j^2 \right)^2 \left( r_{i;1}^2(0)+r_{i;2}^2(0) \right) \left( r_{j;1}^2(0)+r_{j;2}^2(0) \right)
}
{
\left\{ \sum_{k=1}^s b_k^2 \left( r_{k;1}^2(0)+r_{k;2}^2(0) \right) \right\}^2
\sum_{l=1}^s b_l^4 \left( r_{l;1}^2(0)+r_{l;2}^2(0) \right)
}
\\&=
\left\{
\begin{aligned}
&0, &b_1=b_2=\cdots=b_s,\\
&C>0, &\text{else}.
\end{aligned}\right.
\end{align*}

\end{proof}

\par
By Proposition \ref{prop torsion Jordan} (1) and Lemma \ref{lemma M=0 C2}, we proved Proposition \ref{prop torsion Jordan} (2), which completes the proof of Theorem \ref{thm main 2}.

\subsection{Remark}

\

In Theorem \ref{thm main 2} and Proposition \ref{prop torsion Jordan}, the condition that $A$ is invertible cannot be removed. In fact, we have the following two examples.
\par
(1) Let
\begin{align*}
A=
\begin{pmatrix}
0 & 1 & 0 & 0 \\
0 & 0 & 0 & 0 \\
0 & 0 & -1 & 1 \\
0 & 0 & -1 & -1
\end{pmatrix}.
\end{align*}
Then by (\ref{align tor^2}), we have
\begin{align*}
\tau^2(t)
=\frac{\mathrm{e}^{4t}r_2^2(0)}
{ \left\{\mathrm{e}^{2t}r_2^2(0)+r_3^2(0)+r_4^2(0) \right\}^2 }
\end{align*}
for any given $r(0)\in S$. It follows that
\begin{align*}
\lim\limits_{t\to+\infty}\tau(t)=\frac{1}{|r_2(0)|}>0.
\end{align*}
Nevertheless, since $\det A=0$, we cannot obtain stability from $\lim\limits_{t\to+\infty}\tau(t)\neq 0$. In fact, noting that $A$ is a matrix in real Jordan canonical form which has a diagonal block $J_2(0)$, we know that the zero solution of the system is unstable.
\par
(2) Let
\begin{align*}
A=
\begin{pmatrix}
-1 & 1 & 0 & 0 \\
0 & -1 & 1 & 0 \\
0 & 0 & -1 & 0 \\
0 & 0 & 0 & 0
\end{pmatrix}.
\end{align*}
Then by a direct calculation, we have
\begin{align*}
\lim\limits_{t\to+\infty}\tau(t)=+\infty
\end{align*}
for any given $r(0)\in S$.
Nevertheless, since $\det A=0$, the zero solution of the system is not asymptotically stable.

\section{Examples}\label{Section Examples}
In this section, we give two examples, which correspond to Theorem \ref{thm main 1} and Theorem \ref{thm main 2}, respectively.

\subsection{Example 1}

\

Let $r(t)=\left(r_1(t),r_2(t),r_3(t),r_4(t)\right)^\mathrm{T}\in\mathbb{R}^4$, and
\begin{align*}
A=
\begin{pmatrix}
-25 & -8 & -39 & 19 \\[0.7ex]
-14 & -10 & -26 & 14 \\[0.7ex]
9 & 0 & 7 & -9 \\[0.7ex]
-5 & -8 & -21 & -1
\end{pmatrix}.
\end{align*}
Then $\dot{r}(t)=Ar(t)$ is a four-dimensional linear time-invariant system, and
$
\det A=1\, 320\neq0.
$
Set
\begin{align*}
E=\left\{ r(0)\in\mathbb{R}^4 \Bigg| \prod_{i=1}^4 {v_{i}(0)}\neq0 \right\},
\end{align*}
where
\begin{align*}
\left\{
\begin{aligned}
v_{1}(0)&=-r_1(0)-2r_3(0)+r_4(0),\\
v_{2}(0)&=-r_1(0)+2r_2(0)+r_3(0),\\
v_{3}(0)&=-r_1(0)+2r_2(0)+2r_3(0)+r_4(0),\\
v_{4}(0)&=r_1(0)+r_3(0)-r_4(0).
\end{aligned}\right.
\end{align*}
Then the Lebesgue measure of $E$ satisfies
$
m(E)=+\infty.
$
By direct calculation, the limits of the first curvature and the torsion of the trajectory $r(t)$ as $t\to +\infty$ are
$\lim\limits_{t\to+\infty}\kappa(t)=0$
and $\lim\limits_{t\to+\infty}\tau(t)=0$
for $r(0)\in E$, respectively.
Nevertheless, the third curvature $\kappa_3(t)$ of the trajectory $r(t)$ satisfies
\begin{align*}
\lim\limits_{t\to+\infty}\kappa_3(t)=+\infty
\end{align*}
for any $r(0)\in E$.
Consequently, from Theorem \ref{thm main 1}, the zero solution of the system is asymptotically stable.
\par
The graph of the function $\kappa_3(t)$ is shown in Figure \ref{fig k3}, where $r(0)=\left(1,1,1,1\right)^\mathrm{T}$.

\begin{figure}[htbp]
\centering
\includegraphics[height=5cm]{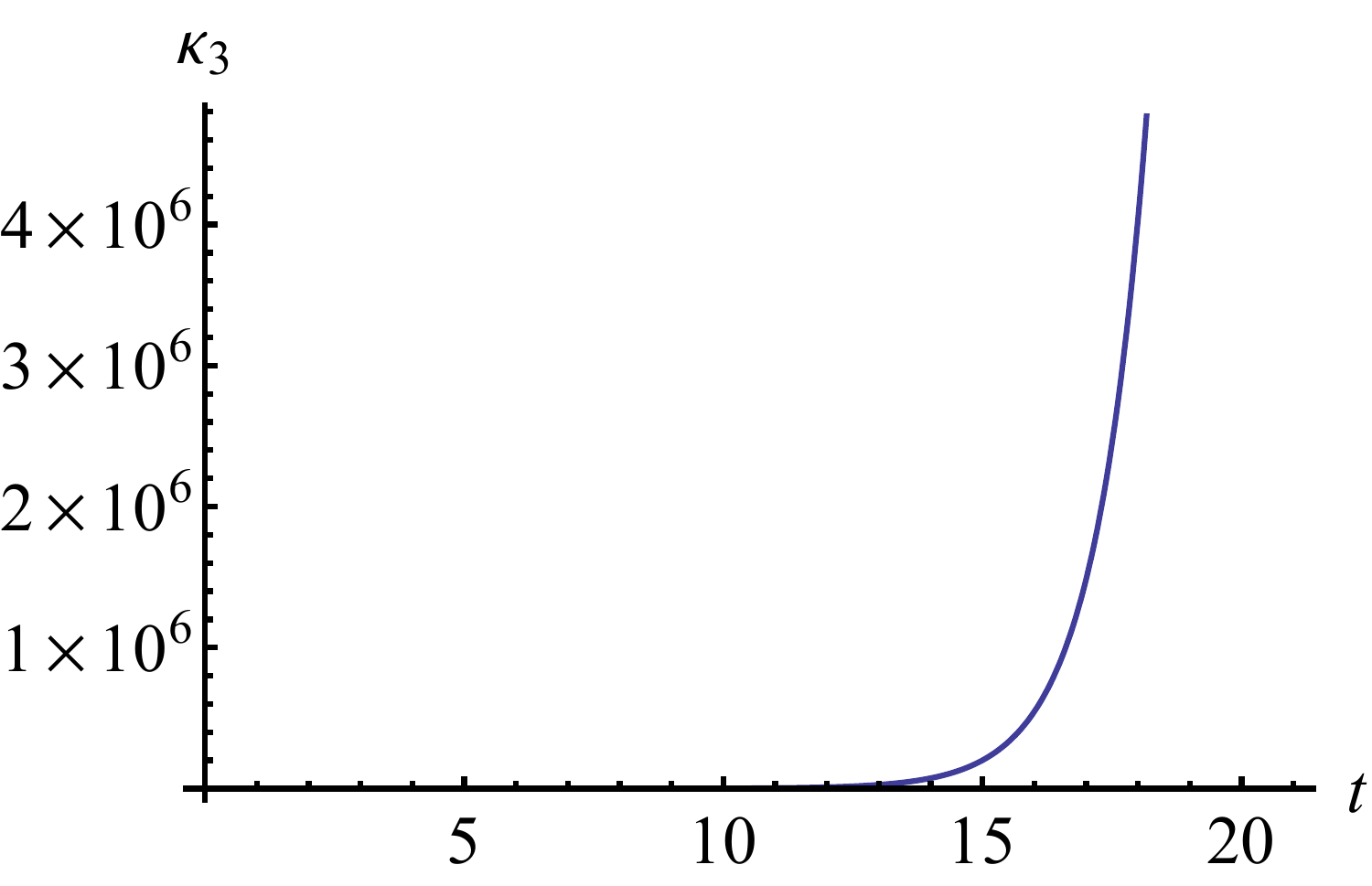}
\caption{Function $\kappa_3(t)$.}
\label{fig k3}
\end{figure}

\subsection{Example 2}

\

We consider a popular model in classical mechanics called coupled oscillators (cf. \cite{Schwartz}). Two masses $P$ and $Q$ are attached with springs. Assume that the masses are identical, i.e. $m_P=m_Q=m$, but the spring constants are different, as shown in the Figure \ref{classical chain}.

%\begin{figure}[htbp]
%\centering
%\includegraphics[height=2cm]{Figure_5_2.eps}
%\caption{1D Coupled Oscillators}
%\label{classical chain}
%\end{figure}

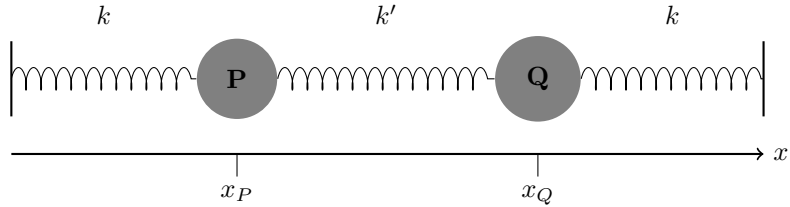
\begin{figure}[htbp]
	\centering
	\begin{tikzpicture}
	\node[circle,fill=gray,inner sep=2.5mm] (a1) at (3,0){$\bf P$};
	\node[circle,fill=gray,inner sep=2.5mm] (a2) at (7,0){$\bf Q$};
	\draw[decoration={aspect=0.3, segment length=2mm, amplitude=1.5mm,coil},decorate](0,0)--(a1)node[midway,above]{$\begin{array}{c}k\\ {}\end{array}$};
	\draw[decoration={aspect=0.3, segment length=2mm, amplitude=1.5mm,coil},decorate](a1)--(a2)node[midway,above]{$\begin{array}{c}k'\\ {}\end{array}$};
	\draw[decoration={aspect=0.3, segment length=2mm, amplitude=1.5mm,coil},decorate](a2)--(10,0)node[midway,above]{$\begin{array}{c}k\\ {}\end{array}$};
	\draw[->,thick](0,-1)--(10,-1) node[right]{$x$};
	\draw (3,-1)--(3,-1.3) node[anchor=north]{$x_P$};
	\draw (7,-1)--(7,-1.3) node[anchor=north]{$x_Q$};
	\draw[thick] (0,-0.5)--(0,0.5);
	\draw[thick] (10,-0.5)--(10,0.5);
	\end{tikzpicture}
	\caption{1D Coupled Oscillators.}
	\label{classical chain}
\end{figure}
\par
Let $x_P$ be the displacement of $P$ from its equilibrium and $x_Q$ be the displacement of $Q$ from its equilibrium. Holding $Q$ fixed and moving $P$, the force on $P$ is
\begin{align*}
F_{1P}=-kx_P-k'x_P.
\end{align*}
Holding $P$ fixed and moving $Q$, the force on $P$ is
\begin{align*}
F_{2P}=k'x_Q.
\end{align*}
Thus by Newton's second law we have
\begin{align*}
m\ddot{x}_P=F_{1P}+F_{2P}=-(k+k')x_P+k'x_Q.
\end{align*}
Similarly, for $Q$ we have
\begin{align*}
m\ddot{x}_Q=-(k+k')x_Q+k'x_P.
\end{align*}
Introducing two variables $v_P=\dot{x}_P$ and $v_Q=\dot{x}_Q$, the above equations are equivalent to the following linear system
\begin{align}\label{align system osc}
\dot{\begin{pmatrix}x_P\\[0.5em] x_Q\\[0.5em] v_P\\[0.5em] v_Q\end{pmatrix}}
=\begin{pmatrix}0& 0&1&0\\[0.5em] 0& 0& 0 & 1\\[0.5em] -\frac{k+k'}{m}&\frac{k'}{m}& 0&0\\[0.5em] \frac{k'}{m}& -\frac{k+k'}{m}& 0&0 \end{pmatrix}
\begin{pmatrix}x_P\\[0.5em] x_Q\\[0.5em] v_P\\[0.5em] v_Q\end{pmatrix}.
\end{align}
For simplicity we denote the system by $\dot{r}(t)=Ar(t)$, where
\begin{align*}
r(t)=\begin{pmatrix}x_P(t)\\[0.5em] x_Q(t)\\[0.5em] v_P(t)\\[0.5em] v_Q(t)\end{pmatrix},
\quad
A=\begin{pmatrix}0& 0&1&0\\[0.5em] 0& 0& 0 & 1\\[0.5em] -\frac{k+k'}{m}&\frac{k'}{m}& 0&0\\[0.5em] \frac{k'}{m}& -\frac{k+k'}{m}& 0&0 \end{pmatrix}.
\end{align*}
\par
Set
\begin{align*}
E=\left\{ r(0)\in\mathbb{R}^4 \Bigg| \left(r_1^2(0)-r_2^2(0)\right)\left(r_3^2(0)-r_4^2(0)\right)\neq0 \right\}.
\end{align*}
Then the Lebesgue measure of $E$ satisfies $m(E)=+\infty.$
By direct calculation, the torsion $\tau(t)$ of the trajectory $r(t)$ is a periodic function and
$
\lim\limits_{t\to+\infty}\tau(t)
$
does not exist
for any $r(0)\in E$.
Hence by Theorem \ref{thm main 2}, the zero solution of the system (\ref{align system osc}) is stable.
\par
As an example, we suppose that $k/m=1$ and $k'/m=2$, and the initial value $r(0)=\left(1,2,1,2\right)^\mathrm{T}$. Then we have
\begin{align*}
\tau^2(t)=
\frac{\sqrt{5} \sin \left(2 \sqrt{5} t\right)+2 \cos \left(2 \sqrt{5} t\right)+17}{2 \left\{\sqrt{5} \sin \left(2 \sqrt{5} t\right)+2 \cos \left(2 \sqrt{5} t\right)-11\right\}^2}.
\end{align*}
The graph of the function $\tau^2(t)$ is shown in Figure \ref{fig tau^2}.

\begin{figure}[htbp]
\centering
\includegraphics[height=5cm]{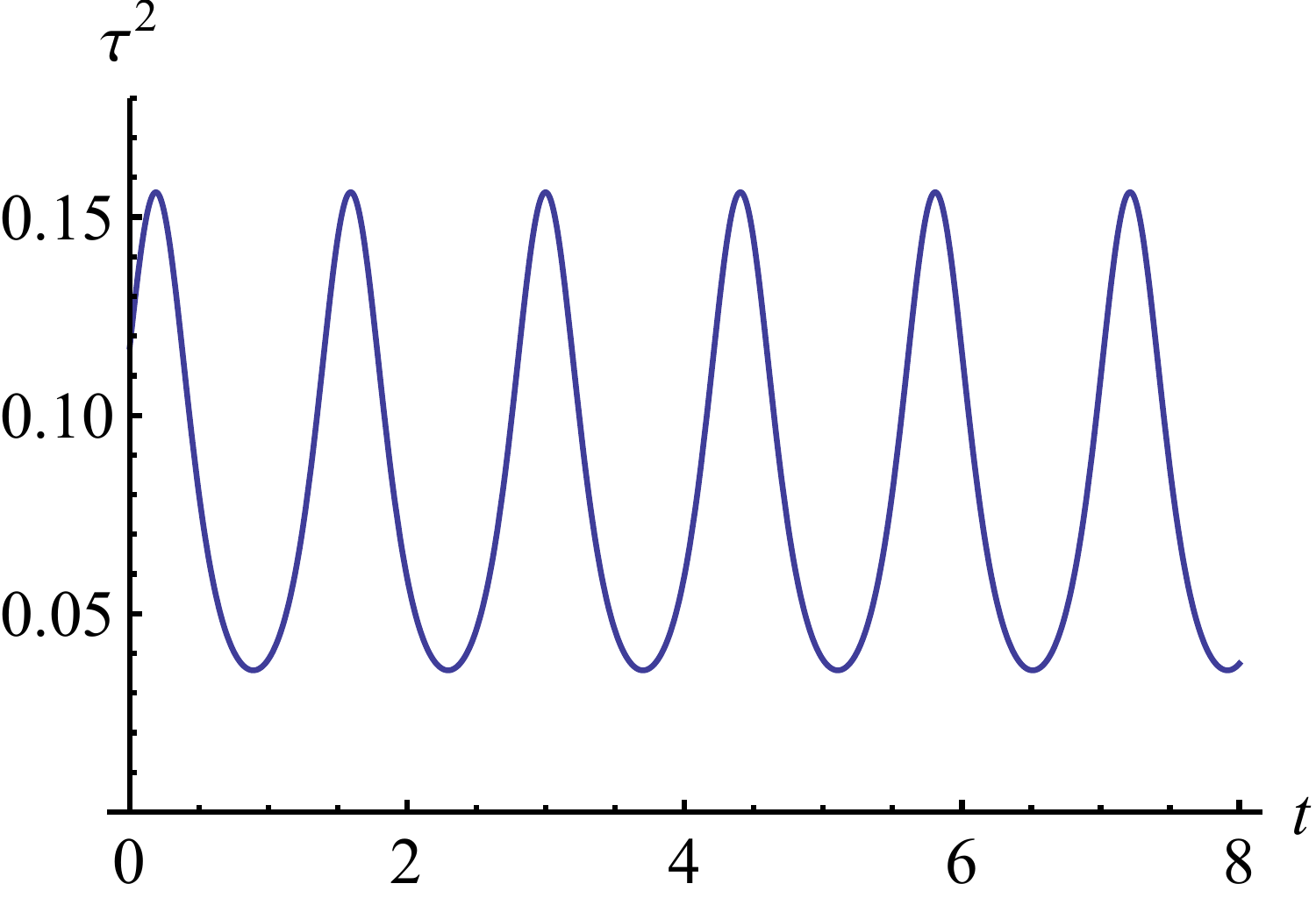}
\caption{Function $\tau^2(t)$.}
\label{fig tau^2}
\end{figure}

\section{Conclusion and Future Work}\label{Section Conclusion}
The main contribution of this paper is to give a geometric description of stability of linear time-invariant systems in arbitrary dimension. Unlike traditional methods based on linear algebra, we focus on the curvature of curves. Specifically, the main results of this paper, Theorem \ref{thm main 1} and Theorem \ref{thm main 2} are proved. For the case where $A$ is similar to a real diagonal matrix, Theorem \ref{thm main 1} gives a relationship between the $i$th curvature $(i=1,2,\cdots)$ of the trajectory and the stability of the zero solution of the system $\dot{r}(t)=Ar(t)$. Further, Theorem \ref{thm main 2} establishes a torsion discriminance for the stability of the system in the case where $A$ is invertible.
\par
For each theorem, we give an example to illustrate the result. In particular, we use the coupled oscillators as an example of the torsion discrimination.
\par
In the future, we will continue to use geometric methods to describe the properties of other kinds of control systems.

\section*{Acknowledgment}
The research is supported partially by science and technology innovation project of Beijing Science and Technology Commission (Z161100005016043).

\end{document}